\documentclass[11pt,reqno,oneside]{amsproc}

\usepackage{amsfonts}
\usepackage{dsfont}

\usepackage{amsmath, amsthm, amssymb, enumerate}
\usepackage{cancel}

\usepackage{color}  
\usepackage{marginnote}
\usepackage[colorlinks=true, pdfstartview=FitV, linkcolor=black, citecolor=black, urlcolor=black]{hyperref}

  \chardef\forshowkeys=0
  \chardef\refcheck=0
  \chardef\showllabel=0
  \chardef\sketches=0
  \chardef\showcolors=0
\ifnum\forshowkeys=1
  
  \usepackage[notref,notcite,color]{showkeys}
\fi

\ifnum\showllabel=1
  \def\llabel#1{\marginnote{\color{colorcccc}\rm\small(#1)}[-0.0cm]\notag}
\else
 \def\llabel#1{\notag}
\fi

\ifnum\refcheck=1
  \usepackage{refcheck}
\fi

\newtheorem{theorem}{Theorem}[section]
\newtheorem{Theorem}{Theorem}[section]

\newtheorem{Lemma}[theorem]{Lemma}

\theoremstyle{definition}

\newtheorem{Remark}[theorem]{Remark}

\def\MM{\tilde M}

\def\andand{\text{\qquad and \qquad}}
\def\uk{u^{(k)}}
\def\ukm{u^{(k-1)}}

\def\umo{u^{(-1)}}
\def\vk{v^{(k)}}
\def\vj{v^{(j)}}

\def\ujm{u^{(j-1)}}

\def\vj{v^{(j)}}

\def\vz{v^{(0)}}

   \def\LLL#1#2{L_{\omega}^{#1}H_{x}^{#2}}

  \def\LLLit{\LLL{\infty}{1/2}}

   \def\Pas{\indeq\mathbb{P}\text{-a.s.}}

   \def\PP{{\mathbb P}}

   \def\ueps{u^{\epsilon}}

   \def\uu{{{u}}}

   \def\startnewsection#1#2{\section{#1}\label{#2}\setcounter{equation}{0}}   
   \def\NNp{{\mathbb N}}
   \def\NNz{{\mathbb N}_0}
   
    \def\TT{{\mathbb T}}
   \def\WW{{\mathbb W}}
   \def\EE{{\mathbb E}}
   \def\comma{ {\rm ,\quad{}} }            
   \def\commaone{ {\rm ,\quad{}} }         
   \def\fractext#1#2{{#1}/{#2}}

   \def\indeq{\qquad{}}                     

\def\TT{\mathbb T}

\def\tilde{\widetilde}

\def\PP{\mathbb{P}}

\def\div{\mathop{\rm div}\nolimits}

\def\indeq{\quad{}}

\ifnum\showcolors=1
  \definecolor{colorcccc}{rgb}{0.7,0.7,0.7}

  \def\colb{\color{black}}
  \definecolor{colorpppp}{rgb}{0.6,0.0,0.1}
  \definecolor{colorgggg}{rgb}{.0,0.4,0.0}
  \definecolor{colorhhhh}{rgb}{0,0.6,0.2}
  \definecolor{colorgray}{rgb}{0.8,0.8,0.8}

  \definecolor{coloroftheorems}{rgb}{0.6,0.0,0.6}
  \definecolor{colorigor}{rgb}{1, 0.2, 0.8}
  \definecolor{amethyst}{rgb}{0.6, 0.4, 0.8}

  \def\cole{\color{coloroftheorems}}
  \def\cole{}
  
  \definecolor{colororange}{rgb}{0.8,0.2,0}
  \definecolor{colorpurple}{rgb}{0.6,0.0,0.6}
  
\else   
  \definecolor{colorcccc}{rgb}{0,0,0}

  \def\colb{\color{black}}
  \definecolor{colorpppp}{rgb}{0,0,0}
  \definecolor{colorgggg}{rgb}{0,0,0}
  \definecolor{colorhhhh}{rgb}{0,0,0}
  \definecolor{colorgray}{rgb}{0,0,0}

  \definecolor{coloroftheorems}{rgb}{0,0,0}
  \definecolor{colorigor}{rgb}{0,0,0}
  \definecolor{amethyst}{rgb}{0,0,0}
  
  \def\cole{\color{coloroftheorems}}
  
  \definecolor{colororange}{rgb}{0.8,0.2,0}
  \definecolor{colorpurple}{rgb}{0.6,0.0,0.6}
  
\fi

  \def\bea{\begin{align}}
  \def\ena{\end{align}}

\def\bega{\begin{aligned}}
  \def\enda{\end{aligned}}

\def\bcase{\begin{cases}}
  \def\ecase{\end{cases}}

\def\bmx{\begin{bmatrix}}
  \def\emx{\end{bmatrix}}

\def\cf{\mathcal{F}}

\def\uu{{{u}}}
\def\WW{ W}

\def\NNp{{\mathbb N}}

\def\lec{\lesssim}

\begin{document}
\baselineskip=12.6pt

$\,$
\vskip1.2truecm
\title[The stochastic Navier-Stokes equations with small $H^{1/2}$ data]{Almost global existence for the stochastic Navier-Stokes equations with small $H^{1/2}$ data}

\author[M.~Aydin]{Mustafa Sencer Ayd\i n}
\address{Department of Mathematics, University of Southern California, Los Angeles, CA 90089}
\email{maydin@usc.edu}

\author[I.~Kukavica]{Igor Kukavica}
\address{Department of Mathematics, University of Southern California, Los Angeles, CA 90089}
\email{kukavica@usc.edu}

\author[F.H.~Xu]{Fanhui Xu}
\address{Department of Mathematics, Union College, Schenectady, NY 12308}
\email{xuf2@union.edu}

\begin{abstract}

We address the global existence of solutions to the stochastic Navier-Stokes equations with multiplicative noise and with initial data in $H^{1/2}(\mathbb{T}^{3})$.  We prove that the solution exists globally in time with probability arbitrarily close to~$1$ if the initial data and noise are sufficiently small.  If the noise is not assumed to be small, then the solution is global on a sufficiently small deterministic time interval with probability arbitrarily close to~$1$.

\hfill 

\today
\end{abstract}

\maketitle

\date{}

\startnewsection{Introduction}{sec01}
Our goal is to address the global existence of solutions to the initial value problem for the stochastic Navier-Stokes equations (SNSE) \begin{align} \begin{split} &\partial_t u - \Delta u + \mathcal{P}((u\cdot\nabla) u) = \sigma(t,u) \dot{W}(t), \\ &\nabla\cdot u = 0,\\ &u(0)=u_0, \end{split} \label{EQ01} \end{align} with small $H^{1/2}$ initial data, where $\mathcal{P}$ denotes the Leray projector and $\sigma(t, u)\dot{W}(t)$ stands for a cylindrical multiplicative It\^o noise.  The SNSE are posed on~$\mathbb{T}^3$, i.e., we consider periodic, zero-average initial data.

The local and global well-posedness of solutions for the \emph{deterministic} Navier-Stokes equations is a classical problem, and there have been many results on local and global existence in critical and subcritical spaces. Note that the Sobolev spaces $W^{k,p}$ are considered critical (resp.~subcritical) if $p(k+1)=3$ (resp.~$p(k+1)>3$).  The local and global existence of solutions in $H^{k}$ spaces for $k\geq 1/2$ has been obtained by Fujita and Kato in~\cite{FK}.  In particular, for $H^{1/2}$ initial data, they established the local existence of solutions for arbitrary (divergence-free) data and the global existence for sufficiently small data.  Additionally, for $L^{p}$ initial data of arbitrary size, when $p\geq3$, Kato in \cite{K} obtained the local existence of solutions, with the global existence for small data.  There are many other results in critical and subcritical spaces; see, for instance,~\cite{CF,FJR,KT,T}.

On the other hand, the local existence in critical and subcritical spaces for the SNSE with multiplicative noise is not as well understood.  Initially, the local existence of solutions was addressed in Bensoussan and Temam (see~\cite{BeT}) with additive noise and in \cite{FRS} with multiplicative noise.  The works \cite{Kr,MR} considered the initial value problem in high regularity Hilbert spaces. The regularity of initial data was then further reduced by Glatt-Holtz and Ziane, who in~\cite{GZ} constructed a maximal strong solution for the equation in 3D bounded domains assuming $H^1$ regularity of the initial data.  In~\cite{Ki}, Kim obtained the local existence of solutions for the 3D SNSE with initial data in $H^{s}$, for $s>1/2$, and the global existence with a large probability for small data.  We also note the results in \cite{AV} by Agresti and Veraar, who obtained the local existence in a range of critical and subcricital Besov spaces.  For an almost global result with a special type of convolution noise and with a vorticity formulation, see~\cite{BR}, while for some other results, see~\cite{BT,BCF,BF,BTr,CC,CFH,DZ,F,FRS,FS,GV,KV,MeS,MoS,R,ZBL} to mention only a few references.

The present paper addresses the global existence with a large probability for the initial value problem with small initial data in~$H^{1/2}$, along with a form of energy inequality.  The main difficulty in the proof lies in the construction of the solution.  When the SNSE is driven by an additive noise, the solutions may be constructed by decomposing it into the stochastic part (an Uhlenbeck-Ornstein-type equation) and the difference, which satisfies a pathwise NSE and can thus be treated using deterministic methods.  In the presence of multiplicative noise, there are issues with convergence, where the use of Poincar\'e inequalities necessitates the subcritical rather than critical initial data.  For this reason, we have to avoid using an approximation procedure and instead employ the idea of infinite decomposition of the data, as introduced by the second and third authors in~\cite{KX2}, when treating the SNSE with small initial data in~$L^{3}$.  Compared to the papers \cite{KX1,KX2,KWX}, the Hilbert space setting allows us to introduce a number of simplifications and the relaxation in the assumptions on the noise. At the same time, since we are taking derivatives, it is more challenging to control the advective term.

One of our main results addresses the almost global existence for data $u_0$ uniformly small in~$H^{1/2}$.  The main idea is to decompose $u_0$ into a series $ u_{0} = v_{0}^{(0)} + v_{0}^{(1)} + v_{0}^{(2)} + \cdots $, which converges in~$H^{1/2}$.  Each term in this decomposition is small in $H^{1/2}$ but possesses $H^{1/2+\delta}$ regularity, where $\delta>0$ is a fixed parameter that can be chosen arbitrarily small.  To achieve this, we use a sequence of cutoffs structurally different than those in \cite{KX1,KX2,KWX}, exploiting the regularizing effect of the dissipation term on $L^2$-based spaces.  This requires care when taking differences and establishing the exponential rate of convergence of the approximations to the solution toward the fixed point.

With $u_0$ decomposed as above and the new cutoffs, we find a unique solution $v^{(k)}$ for a series of equations (see \eqref{EQ32} below), with the initial data $v_0^{(k)}$, such that the sum $ \uk = \vz + \cdots + \vk$ solves the proposed SNSE with the initial data $v_{0}^{(0)} + \cdots + v_{0}^{(k)}$.  Thus, in the limit $k\to\infty$, we expect the solutions $u^{(k)}$ to converge to a solution $u$ of the original SNSE problem.  The critical steps in obtaining the limit are the local persistence of regularity and the nondegeneracy of the limit of the stopping times.

Considering \eqref{EQ32}, we solve for $v^{(k)}$ in $H^{\frac{1}{2}+\delta}$ and establish global-in-time bounds that may grow in $H^{\frac{1}{2}+\delta}$ but stay sufficiently small in $H^{\frac{1}{2}}$ as we sum over~$k$.  Solving \eqref{EQ32} with slightly subcritical initial data necessitates more assumptions on the noise coefficient, such as \eqref{EQ04} with $\alpha = \delta$.  Since \eqref{EQ32} is a SNSE type system with forcing, solving it in $H^\frac{1}{2}$ is as challenging as our original problem.  Therefore, we have to consider each piece $v^{(k)}_0$ of the original initial datum in $H^{\frac{1}{2}+\delta}$.  However, we may still replace the subcritical smallness assumption on the noise coefficient with a critical smallness and a subcritical boundedness assumption; see Remark~\ref{R04}.

We would also like to point out the main feature of our second main result.  Namely, its proof and Remark~\ref{R01} show how to obtain the global existence with large probability on any interval $[0,T]$ without requiring the noise coefficient to be small.

The paper is structured as follows. In Section~\ref{sec02}, we introduce the setting and state the main results, Theorem~\ref{T01}, on almost global solutions with small data and small noise, and Theorem~\ref{T02}, providing a local solution with small data without the smallness assumption on the noise.  In Section~\ref{sec.heat}, we introduce the main auxiliary result, Lemma~\ref{L01}, on the local solution for the stochastic heat equation.  Initial steps for the proofs are provided in Section~\ref{sec04}. In particular, Lemma~\ref{L05} contains the decomposition of the initial data, the main inductive construction, and the pointwise energy estimates needed for global energy control.  The proofs are concluded in Section~\ref{sec.thm}.

\startnewsection{Preliminaries and the main results}{sec02}

We recall that, when $u$ is sufficiently smooth, the Leray projector $\mathcal{P}$ on $\mathbb{T}^3$ is given by
\begin{align}
( \mathcal{P} \uu)_j( x)=\sum_{k=1}^{3}( \delta_{jk}+R_j R_k) \uu_k( x)\comma j=1,2,3,
\llabel{EQ02}
\end{align} 
with $R_j$'s representing the Riesz transforms. We fix
$\delta>0$,
which can be arbitrarily small. 
Consider a cylindrical multiplicative noise,
denoted by $\sigma$, for which we assume that
\begin{align}
	& \Vert\sigma(t, u_1)-\sigma(t, u_2)\Vert_{\mathbb{H}^{\frac12 +\alpha}} 
	 \le \epsilon_{\sigma}\Vert  u_1-u_2\Vert_{H^{\frac12 +\alpha}}
	  \comma t\ge 0
	  \commaone u_1, u_2\in H^{\frac12 +\alpha}(\TT^3)
	  \commaone
	   \alpha=0,\delta,\label{EQ04}
	\\&
	\sigma(t,0)=0 \comma t\ge 0
	\label{EQ05}
	,
  \end{align}
where $\epsilon_{\sigma}>0$, and
$\mathbb{H}^{\frac12 +\alpha}$-norm is defined in~\eqref{EQ09}
below.
As stated in the introduction, the smallness assumption in \eqref{EQ04} for $\alpha =
\delta$ is not essential and can be relaxed; see Remark~\ref{R04} below.
Without loss of generality, we assume that $\div (\sigma(t, u))=0$
and $\int_{\mathbb{T}^{3}}\sigma(t,u)=0$
for $t\geq 0$ when
$\div u=0$
and $\int_{\mathbb{T}^{3}}u=0$.
Otherwise, we may absorb $(I-\mathcal{P})\sigma$ into the pressure term, and remove it using the
Leray
projector~$\mathcal{P}$.
In addition, we require that $\int_{\TT^3} \sigma (t, u)\, dx = 0$ for
every~$t\geq0$ when
$\div u=0$ and
$\int_{\TT^3} u\, dx = 0$. 
We note that our assumptions allow for a linear-structured noise such as  
  \begin{align}
     \sigma(t, u) = \frac{1}{C}\epsilon_{\sigma}\mathcal{P}_0(\phi*P(t,u))
  ,\llabel{EQ06}
\end{align}
where $\phi \in C^\infty(\TT^3)$, $P(t, \cdot)$ is an operator for which \eqref{EQ04} and~\eqref{EQ05} hold, $C>0$ is a sufficiently large constant, and
$\mathcal{P}_0$ is the average-free Leray projection. Namely, $\mathcal{P}_0 f = \mathcal{P}\left( f- \left(\int f\right/|\mathbb{T}^3|)\right)$. 

We fix a stochastic basis $(\Omega, \mathcal{F},(\mathcal{F}_t)_{t\geq 0},\mathbb{P})$
that satisfies the standard assumptions, and denote by $\mathcal{H}$ and $\mathcal{Y}$ separable Hilbert spaces.
We consider $\{\mathbf{e}_k\}_{k\geq 1}$, a complete orthonormal basis of~$\mathcal{H}$, and note that
$\WW( t,\omega):=\sum_{k\geq 1} W_k( t,\omega) \mathbf{e}_k$ yields a cylindrical Wiener process over~$\mathcal{H}$,
where $\{W_k: k\in\NNp\}$ is a family of independent Brownian motions in $(\Omega, \mathcal{F},(\mathcal{F}_t)_{t\geq 0},\mathbb{P})$.
We write $l^2( \mathcal{H},\mathcal{Y})$ for
the space of Hilbert-Schmidt operators with the norm 
\begin{align}
\Vert G\Vert_{l^2( \mathcal{H},\mathcal{Y})}^2:= \sum_{k=1}^{\dim \mathcal{H}} | G \mathbf{e}_k|_{\mathcal{Y}}^2.
\llabel{EQ07}
\end{align}
In this setting, the Burkholder-Davis-Gundy inequality reads
\begin{align}
\EE \biggl[ \sup_{s\in(0,t]}\biggl| \int_0^s G \,d\WW_r \biggr|_{\mathcal{Y}}\biggr]
\leq 
\EE\biggl[ \left(\int_0^t \Vert G\Vert^2_{ l^2( \mathcal{H},\mathcal{Y})}\, dr \right)^{1/2}\biggr].
\llabel{EQ10}
\end{align} 

For a real number $\alpha$ and the Sobolev space $H^{\alpha}=H^{\alpha}(\mathbb{T}^3)$, we write
  \begin{align}
   \mathbb{H}^{\alpha}=\left\{f\colon\TT^3\to l^2( \mathcal{H},\mathcal{Y})
   : 
  f  \mathbf{e}_k\in H^{\alpha}(\TT^3), k\in \mathbb{N}, \text{ and } \int_{\TT^3} \bigl\Vert \bigl((1+|n|^2)^{\alpha/2}\hat{f}\bigr)^{\vee}\bigr\Vert_{l^2( \mathcal{H},\mathcal{Y})}^2 \,dx<\infty
  \right\}
  \llabel{EQ08}
  \end{align}
for the Sobolev norms on noise coefficients. We note that 
$\left(\mathbb{H}^{\alpha},\Vert \cdot\Vert_{\mathbb{H}^{\alpha}}\right)$ is a Hilbert space with the norm
\begin{align}
\Vert f\Vert_{\mathbb{H}^{\alpha}}:=\left( \int_{\TT^3} \big\Vert \big((1+|n|^2)^{\alpha/2}\hat{f}\big)^\vee\big\Vert_{l^2( \mathcal{H},\mathcal{Y})}^2 \,dx\right)^{\frac{1}{2}}. 
\label{EQ09}
\end{align}
We write $\mathbb{L}^{2}$ instead of~$\mathbb{H}^{0}$ and set
$( \mathcal{P}f) \mathbf{e}_k=\mathcal{P} (f \mathbf{e}_k)$.
It follows that
$\mathcal{P}f\in \mathbb{H}^{\alpha}$ for $ f\in \mathbb{H}^{\alpha}$.
We also use $C>0$ to denote a sufficiently large constant
that may change from line to line.

Now, we are ready to state our main result.

\cole
\begin{Theorem}[Global solution with small noise and initial data]
\label{T01}
Let $\uu_0\in L^\infty(\Omega; H^{1/2}(\TT^3))$ be such that
$\nabla\cdot u_0=0$ and $\int_{\TT^3} u_0=0$.
Suppose that the assumptions~\eqref{EQ04} and \eqref{EQ05} hold,
where $\epsilon_{\sigma}\in(0,1]$ is sufficiently small.
For every $p_0\in(0,1]$, there exists $\epsilon_0\in(0,1]$
such that if 
  \begin{align}
    \sup_{\Omega}\Vert u_0\Vert_{H^\frac12}\leq \epsilon_0
    ,
   \label{EQ11}
  \end{align}
then there is
a stopping time
$\tau\in(0,\infty]$ and a unique solution $(u, \tau)$
of 
\eqref{EQ01} on $(\Omega, \mathcal{F},(\mathcal{F}_t)_{t\geq 0},\mathbb{P})$,
with the initial condition
$u_0$, such that
  \begin{align}
  \begin{split}
    \EE\biggl[
    \sup_{0\leq t\leq \tau}
    \Vert\uu(t,\cdot)\Vert_{H^{\frac12}}^2
    +\int_0^{\tau} 
     \Vert\uu(t,\cdot)\Vert_{H^{\frac{3}{2}}}^2 \,dt
    \biggr]
    \leq 
    C\epsilon_0^{2}
  \end{split}
   \label{EQ12}
  \end{align}
and
  \begin{align}
   \mathbb{P}[\tau<\infty]
   \leq p_0.
   \llabel{EQ13}
  \end{align}
\end{Theorem}
\colb

We note that our solution is probabilistically strong. Specifically, 
we consider $u$, a progressively measurable, divergence-free process in $L^2(\Omega; C([0,\tau], H^{1/2}))$,
and $\tau$, a stopping time on
$(\Omega, \cf, (\cf_t)_{t\geq 0}, \PP)$.
We say that $(\uu,\tau)$ solves \eqref{EQ01} if
  \begin{align}
   \begin{split}
     (u_j( t \wedge \tau),\phi)
     &=
     (u_{j, 0},\phi)
   + \int_{0}^{t\wedge\tau}
         (u_j (r), \Delta\phi )
     \,dr
   \\&\indeq
   + \int_{0}^{t\wedge\tau}
      \bigl(\bigl(\mathcal{P}  (u_m(r) u(r))\bigr)_j ,\partial_{m} \phi\bigr)
     \,dr
   \\&\indeq
     +\int_0^{t\wedge\tau} \bigl(\sigma_j(r, \uu(r)),\phi\bigr)\,dW(r)\quad
     \PP\mbox{-a.s.}
    \comma j=1,2,3
    ,
   \end{split}
   \label{EQ14}
  \end{align}
for all $\phi\in C^{\infty}(\TT^3)$ and $t\in [0,\infty)$,
where $(f,g)$ denotes the $L^2(\mathbb{T}^3)$ inner product of $f$ and $g$. 
\colb

The conclusions of Theorem~\ref{T01} still hold when $\uu_0\in L^1(\Omega; H^\frac12 (\TT^3))$,
allowing $\tau\in[0,\infty]$. 
By Markov's inequality, if the $L^1(\Omega; H^\frac12 (\TT^3))$-norm
of $\uu_0$ is sufficiently small, there exists a sufficiently large subspace $\Omega_1\subset \Omega$ such that $u_0 \mathds{1}_{\Omega_1}$ satisfies \eqref{EQ11} for a sufficiently small $\epsilon_0\in(0,1]$. We can then apply the current version of Theorem~\ref{T01} and conclude the existence of a unique solution $(u, \tau)$ to~\eqref{EQ01} with the initial datum $u_0 \mathds{1}_{\Omega_1}$. Finally, by setting $\tau=0$ and $u(0)=u_0$ on~$\Omega_1^{\text{c}}$, we deduce that $(u, \tau)$ solves~\eqref{EQ01} with the initial datum~$\uu_0$.

Removing the assumption that $\epsilon_{\sigma}>0$ is small, 
we establish the local existence and
pathwise uniqueness of the solution to~\eqref{EQ01}.
 
\cole
\begin{Theorem}[Local solution with small initial data]
\label{T02}
Let $T>0$, and assume that
$\uu_0\in L^\infty(\Omega; H^\frac12 (\TT^3))$
satisfies $\nabla\cdot u_0=0$ and $\int_{\TT^3} u_0=0$. Suppose that the assumptions~\eqref{EQ04} and \eqref{EQ05} hold, where~$\epsilon_{\sigma}>0$.
Then, there exists 
$\epsilon_0\in(0,1]$
such that if 
  \begin{align}
   \sup_{\Omega} \Vert u_0\Vert_{H^\frac12}\leq \epsilon_0
    ,
   \llabel{EQ15}
  \end{align}
there is
a stopping time
$\tau>0$~a.s.\ and a unique solution $(u,\tau)$
of 
\eqref{EQ01} on $(\Omega, \mathcal{F},(\mathcal{F}_t)_{t\in[0,T]},\mathbb{P})$, with the initial condition
$u_0$, such that
  \begin{align}
  \begin{split}
    \EE\biggl[
    \sup_{0\leq t\leq \tau}\Vert\uu(t,\cdot)\Vert_{H^\frac12}^2
    +\int_0^{\tau} 
    \Vert\uu(t,\cdot)\Vert_{H^{\frac{3}{2}}}^2 \,dt
    \biggr]
    \leq 
    C\epsilon_0^{2}
    ,
  \end{split}
   \label{EQ16}
  \end{align}
for some positive constant~$C$.
\end{Theorem}
\colb

To prove Theorems~\ref{T01} and~\ref{T02}, we first
collect some results regarding the stochastic heat equation in Section~\ref{sec.heat}. 
Next, in Section~\ref{sec04}, we approximate the initial data
by a sum of smooth functions, treating each term in this sum as the initial datum for an SNSE-type
system; see~\eqref{EQ32}. Then, upon truncation, we establish the existence of a unique solution and derive energy bounds for these systems. Moreover, we show that the stopping times up to which each solution exists are almost surely positive and finite with arbitrarily small probability. Finally, in Section~\ref{sec.thm},
we conclude the proofs of our main theorems.

\startnewsection{Stochastic heat equation}{sec.heat}
\colb
Our results rely on an analysis of the 
stochastic heat equation 
\begin{align}
  \begin{split}
    \partial_t\uu( t,x)
    &=\Delta \uu( t,x)
       + f( t,x)
       + g( t,x)\dot{\WW}(t),
    \\
    \uu( 0,x)&= \uu_0 ( x) \Pas
    ,
  \end{split}
  \label{EQ17} 
\end{align}
where 
$u\colon \Omega\times[0,T]\times\mathbb{T}^{3}\to \mathbb{R}$, for some $T>0$. In the sequel, $\Lambda$ denotes the operator corresponding to the Fourier multiplier~$(1+|n|^2)^{\alpha/2}$ for $n \in \mathbb{Z}^3$. 
In the next lemma, we provide estimates for the solutions of stochastic heat equation
over the nonhomogeneous Sobolev spaces in which we construct our solution.

\cole
\begin{Lemma}
\label{L01}
Let $T\in(0,\infty)$
and $\alpha\in \{0,\delta\}$.
Suppose that $u_0\in L^2(\Omega, H^{\frac{1}{2}+\alpha})$,
$f\in L^2(\Omega\times[0,T], H^{\alpha -\frac12})$, 
and $g\in L^2(\Omega\times[0,T], \mathbb{H}^{\frac12 +\alpha})$
are divergence and average-free.
Then there exists a unique solution 
$u\in L^2(\Omega; C([0,T], H^{\frac12 +\alpha}))\cap L^2(\Omega; L^2([0,T], H^{\frac{3}{2}+\alpha}))$
of \eqref{EQ17}
such that
  \begin{align}
    \begin{split}
      \EE\biggl[&
         \sup_{0\leq t\leq T}
        \Vert\uu(t)\Vert_{H^{\frac12 +\alpha}}^2
	 - \Vert\uu_0\Vert_{H^{\frac12 +\alpha}}^2
	+ \int_0^{T} \Vert\uu(t)\Vert_{H^{\frac{3}{2} +\alpha}}^2 \,dt\biggr]
      \\&\indeq
      \leq C
      \EE\biggl[
      \int_0^{T}\Vert f(s)\Vert_{H^{\alpha- \frac12}}^2
           \,dt
      +
      \int_0^{T}\Vert g(s)\Vert_{\mathbb{H}^{\frac12 +\alpha}}^2
           \,dt
      \biggr]
      ,
    \end{split}
    \label{EQ18}
  \end{align}
where $C>0$ is independent of $T$ and $u$ is average-free.
\end{Lemma}
\colb

\begin{proof}[Proof of Lemma~\ref{L01}]
 
We provide a brief explanation of the argument for obtaining a solution that satisfies \eqref{EQ18}; see~\cite{KXZ} or~\cite{Ki} for more details. First,
we mollify $f$, $g$, and $u_0$ in~\eqref{EQ17}
to $f^{\epsilon}$, $g^{\epsilon}$, and $u_0^{\epsilon}$,
obtaining a sequence of spatially regular solutions~$\ueps$ (see~\cite[Theorem 3.1]{LR}). Note that $\Lambda^{\frac12 +\alpha}\ueps$ is a well-defined function due to the regularity of~$\ueps$; moreover, it satisfies 
\begin{align}
	\begin{split}
		\partial_t(\Lambda^{\frac12 +\alpha}\ueps)( t,x)
		&=\Delta (\Lambda^{\frac12 +\alpha}\ueps)( t,x)
		+ (\Lambda^{\frac12 +\alpha}f^{\epsilon})( t,x)
		+ (\Lambda^{\frac12 +\alpha}g^{\epsilon})( t,x)\dot{\WW}(t),
		\\
		(\Lambda^{\frac12 +\alpha}\ueps)( 0,x)&= (\Lambda^{\frac12 +\alpha}\ueps_0) ( x) \Pas
	\end{split}
	\llabel{EQ17-2} 
\end{align}
Then, by It\^o's formula, we have
 \begin{align}
  \begin{split}
   &\Vert\Lambda^{\frac12 +\alpha}\ueps(t)\Vert_{L^2}^{2}
   - \Vert\Lambda^{\frac12 +\alpha}\ueps_0\Vert_{L^2}^{2} 
   + 2\int_0^t \int_{\TT^3} | \Lambda^{\frac{3}{2}+\alpha} \ueps(r)|^2 \,dx dr
  \\&\indeq
  \leq
  2\int_0^t\left|  (f^{\epsilon}(r),\Lambda^{1+2\alpha}\ueps)\right| \,dr
\\&\indeq\indeq
+ \int_0^t\int_{\TT^3} \Vert \Lambda^{\frac12 +\alpha}g^{\epsilon}(r)\Vert_{l^2}^2\,dx dr
+ 2 
\left|\int_0^t( g^{\epsilon}(r),\Lambda^{1+2\alpha}\ueps(r)) \,d\WW_r\right|
\\&\indeq
= 
 I_1 + I_2 + I_3
\commaone t\in [0,T].    
\end{split}
   \label{EQ19}
\end{align}

The term $I_2$ is bounded by the $L^2([0,T], \mathbb{H}^{\frac12 +\alpha})$-norm of $g^{\epsilon}$ almost surely. For $I_1$, we use
\begin{align}
  \begin{split}
   (\Lambda^{\alpha- \frac12}f^{\epsilon},\Lambda^{\frac{3}{2}+\alpha}u^{\epsilon})
   &\leq
   \Vert \Lambda^{\alpha- \frac12}f^{\epsilon}\Vert_{L^2}
   \Vert \Lambda^{\frac{3}{2}+\alpha}u^{\epsilon}\Vert_{L^2}
   \leq
   \eta \Vert \Lambda^{\frac{3}{2}+\alpha}u^{\epsilon}\Vert_{L^2}^{2}
   +
   C_{\eta}
   \Vert \Lambda^{\alpha- \frac12}f^{\epsilon}\Vert_{L^2}^{2},
  \end{split}
   \label{EQ20}
  \end{align}
where $\eta>0$ is arbitrary.
The spatial integral in $I_3$ has the same structure,
so we estimate it similarly
  \begin{align}
  \begin{split}
  	\EE&\left[ 
  	 \sup_{0\leq t\leq T}
  	  \left|\int_0^t ( g^{\epsilon}(r),\Lambda^{1+2\alpha}\ueps(r)) d\WW_r\right|
  	\right]
  	\le 
  	\EE\left[
  	 \biggl( 
  	   \int_0^T
  	    \left\Vert ( \Lambda^{\alpha-\frac{1}{2}}g^{\epsilon},\Lambda^{\frac{3}{2}+\alpha}\ueps)
  	     \right\Vert^2_{\mathbb{L}^2} \,dr
  	   \biggr)^\frac12 \right]
  \\&\leq  \EE\left[ \int_0^T \left(
  \eta \Vert \Lambda^{\frac{3}{2}+\alpha}u^{\epsilon}(r)\Vert_{L^2}^{2}
  +
  C_{\eta}
  \Vert \Lambda^{\alpha- \frac12}g^{\epsilon}(r)\Vert_{\mathbb{L}^2}^2 \right)\,dr  \right]
   ,
  \end{split}
   \label{EQ21}
  \end{align} 
where we used the Burkholder-Davis-Gundy inequality in the first step.
Taking the supremum over $t$ and the expectation on both sides of \eqref{EQ19}, then employing \eqref{EQ20}--\eqref{EQ21} with a sufficiently small $\eta$, we arrive at 
\begin{align}
	\begin{split}
		&\EE\biggl[
		\sup_{0\leq t\leq T}
		\Vert\uu^{\epsilon}(t,\cdot)\Vert_{H^{\frac12 +\alpha}}^2
		- \Vert\uu^{\epsilon}_0\Vert_{H^{\frac12 +\alpha}}^2
		+ \int_0^{T}\int_{\mathbb{T}^{3}} | \Lambda^{\frac{3}{2}+\alpha}u^{\epsilon}|^2\,dx \,dt\biggr]
		\\&\indeq
		\leq C
		\EE\biggl[
		\int_0^{T}\Vert \Lambda^{\alpha-\frac12}f^{\epsilon}(s,\cdot)\Vert_{L^2}^2
		\,dt
		+
		\int_0^{T}\Vert \Lambda^{\alpha- \frac12}g^{\epsilon}(s,\cdot)\Vert_{\mathbb{L}^2}^2
		\,dt
		+\int_0^{T} \Vert \Lambda^{\frac12 +\alpha}g^{\epsilon}(s, \cdot)\Vert_{\mathbb{L}^{2}}^{2}  \,dt
		\biggr]
		.
	\end{split}
	\label{EQ22}
\end{align} 
Then, upon utilizing the Poincar\'e inequality, we conclude~\eqref{EQ18} for~$\ueps$. Finally, 
we may apply It\^o's formula to $u^\epsilon-u^{\epsilon'}$ and repeat the computations above 
obtaining a bound akin to \eqref{EQ22} involving $f^\epsilon-f^{\epsilon'}$, $g^\epsilon-g^{\epsilon'}$, and 
$u_0^\epsilon-u_0^{\epsilon'}$.
Utilizing the resulting inequalities, these functions converge to $0$ as $\epsilon,\epsilon' \to 0$,
we obtain a limit
$u\in L^2(\Omega; C([0,T], H^{\frac12 +\alpha}))\cap L^2(\Omega; L^2([0,T],H^{\frac{3}{2}+\alpha}))$
of $\ueps$ as $\epsilon\to 0$.
This function~$u$ obeys the energy inequality~\eqref{EQ18}, solves the stochastic heat equation in its strong formulation~\eqref{EQ14}, and remains average-free under the equation~\eqref{EQ17} and the assumptions on $u_0$, $f$, and~$g$.
\end{proof}

When addressing the stochastic Navier-Stokes-like equations, we shall apply this lemma with $f=\mathcal{P}((v\cdot\nabla) w)$. The Sobolev product inequality, the Sobolev embedding inequality, and the interpolation inequality yield
 \begin{align}
		 	\begin{split}
			 	\Vert v \otimes w\Vert_{H^{\frac{1}{2}+\alpha}}^2
			 	 &\lec
		 	  \Vert  v\Vert_{W^{\frac{1}{2}+\alpha, \frac{6}{2+\alpha}}}^2
		 	   \Vert w\Vert_{L^{\frac{6}{1-\alpha}}}^2
			 	   +\Vert v\Vert_{L^{\frac{6}{1-\alpha}}}^2
			 	   \Vert w\Vert_{W^{\frac{1}{2}+\alpha, \frac{6}{2+\alpha}}}^2
			 	  \\& \lec
			 	  \Vert v\Vert_{H^{1+\frac{\alpha}{2}}}^{2}
			 	  \Vert w\Vert_{H^{1+\frac{\alpha}{2}}}^{2}
			 	  \lec
			 	    \Vert v\Vert_{H^{\frac{1}{2}+\alpha}}^{1+\alpha}
			 	    \Vert v\Vert_{H^{\frac{3}{2}+\alpha}}^{1-\alpha}
			 	    \Vert w\Vert_{H^{\frac{1}{2}+\alpha}}^{1+\alpha}
			 	     \Vert w\Vert_{H^{\frac{3}{2}+\alpha}}^{1-\alpha}
,
			 	    \end{split}
		 	    \label{EQ23}
		 \end{align}
for $\alpha\in[0,1)$. Therefore, if $v=w$ and $\alpha=0$, we have
  \begin{align}
  \begin{split}
   \Vert \Lambda^{-\frac{1}{2}}f\Vert_{L^2}^2
   &\leq
   \Vert v\otimes v\Vert_{H^{\frac{1}{2}}}^2
   \lec
   \Vert v\Vert_{H^\frac12}^2
   \Vert v\Vert_{H^\frac{3}{2}}^2
  \end{split}
   \label{EQ24}
  \end{align}
under the incompressibility condition; if $\alpha\neq 0$, we still have
  \begin{align}
  \begin{split}
   \Vert \Lambda^{-\frac12+\alpha}f\Vert_{L^2}^2
   \leq
   \Vert v\otimes v\Vert_{H^{\frac{1}{2}+\alpha}}^2
   \lec
   \Vert v\Vert_{H^\frac12}^2
   \Vert v\Vert_{H^{\frac{3}{2}+\alpha}}^2
  \end{split}
   \label{EQ25}
  \end{align}
by interpolation.

\startnewsection{Global solutions of the truncated difference equations}{sec04}

\subsection{Decomposition into smooth data and a construction of $H^{\frac12 +\delta}$ solutions}

We consider a divergence-free and average-free initial datum $u_0$ that satisfies \eqref{EQ11}, with a fixed $\epsilon_0>0$. The initial datum $u_0$ shall be approximated by partial sums of a sequence of smooth functions.

\cole
\begin{Lemma}[Decomposition of initial data]
\label{L05}
There exists a sequence 
  \begin{align}
   v_{0}^{(0)},v^{(1)}_{0},v_{0}^{(2)},\ldots \in L^\infty(\Omega;C^\infty(\mathbb{T}^3))
   \label{EQ26}
  \end{align}
consisting of divergence-free functions with zero-average, which satisfy
  \begin{align}
   \sup_{\Omega} \Vert   v_{0}^{(0)}\Vert_{H^\frac12 (\mathbb{T}^3)}
   \leq 2\epsilon_0
   \label{EQ27}
  \end{align}
and
  \begin{align}
   \sup_{\Omega}\Vert   v_{0}^{(k)}\Vert_{H^\frac12 (\mathbb{T}^3)}
   \leq \frac{\epsilon_0}{4^{k}}
    \comma k=1,2,3,\ldots
    .
   \label{EQ28}
  \end{align}
Moreover,
  \begin{align}
   u_{0} = v_{0}^{(0)} + v_{0}^{(1)} + v_{0}^{(2)} + \cdots \quad\mbox{ in }H^{\frac12}(\mathbb{T}^3)
   ,
   \label{EQ29}
  \end{align}
almost surely.
\end{Lemma}
\colb

The proof is analogous to that of~\cite[Lemma~3.1]{KX2} and is therefore omitted.

Now, fix $p_0\in(0,1)$, and let $\epsilon_0\in (0, 1/2)$ (to be determined; see the proofs of Lemma~\ref{L10} and Theorem~\ref{T01}).
Additionally, decompose the initial datum $u_0$ into a sequence of functions as in
\eqref{EQ26}, ensuring that \eqref{EQ27}--\eqref{EQ29} are satisfied. Furthermore, let $\MM_0, \MM_1, \ldots$ be such that
  \begin{align}
   \sup_{\Omega}\Vert   v_{0}^{(k)}\Vert_{H^{\frac12 +\delta}}
   \leq \MM_k
   ,
   \label{EQ30}
  \end{align}
for all $k\in\NNz$.
We construct our solution by considering 
  \begin{align}
	\umo=0
	\andand
	\uk = \vz + \cdots + \vk
	\comma k\in\NNz
	 ,
	\label{EQ31}
\end{align}
where $\vk$ solves the incompressible Navier-Stokes-like system
  \begin{align}
	\begin{split} 
		\partial_t \vk  -\Delta \vk
		&=
		-
		\mathcal{P}\bigl( \vk\cdot \nabla \vk\bigr)
		-
		\mathcal{P}\bigl( \ukm\cdot \nabla \vk\bigr)
		-
		\mathcal{P}\bigl( \vk\cdot \nabla \ukm\bigr)
		\\&\indeq
		+
		\sigma(t, \uk( t,x))\dot{\WW}(t)
		-
		\sigma(t, \ukm( t,x))\dot{\WW}(t),
		\\   \nabla\cdot v^{(k)}( t,x) &= 0
		,
		\\
		v^{(k)}( 0,x)&=v_{0}^{(k)} (x)  \Pas
		\commaone x\in\TT^3 
		.
	\end{split}
	\label{EQ32}
\end{align}
It follows from \eqref{EQ31}--\eqref{EQ32} that $\uk$ 
is a solution to the
SNSE with the initial data
$v_{0}^{(0)} + \cdots + v_{0}^{(k)}$ over the interval where $v^{(j)}$, for $j=0,\ldots, k$, all exist; see~\eqref{EQ119} and~\eqref{EQ120} below.

To obtain $v^{(k)}$, we employ cutoff functions
defined as follows.
First, we choose a smooth function
$\theta\colon[0,\infty)\to [0,1]$ such that $\theta\equiv 1$ on $[0,1]$ and 
$\theta\equiv 0$ on~$[2,\infty)$. For each $k \in \mathbb{N}_0$, we introduce 
  \begin{align}
	\psi_k(v)
	=\psi
	:=
	\theta\left(
	 \frac{1}{M_k}
	 \Vert v(t)\Vert_{H^{\frac{1}{2}+\delta}}
	 + \frac{1}{M_k}\left(\int_0^t \Vert v(s)\Vert_{H^{{\frac{3}{2}}+\delta}}^2\,ds\right)^\frac12
	  \right)
	\label{EQ33}
\end{align}
and
\begin{align}
	\phi_k(v)=
	\phi
	:=
	\theta\left(\frac{2^{k}}{\bar\epsilon}\Vert v(t)\Vert_{H^\frac12}
	   +\frac{2^{k}}{\bar\epsilon}\left(\int_0^t \Vert v(s)\Vert_{H^{\frac{3}{2}}}^2\,ds\right)^\frac12\right)
	,\label{EQ34}
\end{align}
where $\bar{\epsilon}\in (2\epsilon_0, 1)$, and $\{M_k\}_{k\in\mathbb{N}_0} \in \mathbb{R}^+$
is a nondecreasing sequence to be determined; see the proofs of Lemma~\ref{L10} and Theorem~\ref{T01}.
Above, we abbreviated $v=v^{(k)}$.
We emphasize that, in contrast to \cite{KX2}, we have incorporated
the Sobolev energy dissipation integral into the cutoff.
Then, we study the system
\begin{align}
	\begin{split}
		\partial_t v  -\Delta v
		&=
		-
		\psi^2 \phi^2
		\mathcal{P}\bigl( v\cdot \nabla v\bigr)
		-
		\psi^2 \phi^2 \zeta
		\mathcal{P}\bigl( w\cdot \nabla v\bigr)
		-
		\psi^2 \phi^2 \zeta
		\mathcal{P}\bigl( v\cdot \nabla w\bigr)
		\\&\indeq\indeq
		+
		\psi^2 \phi^2 \zeta\left( 
		\sigma(t, v+w) - \sigma(t, w)
		\right)
		\dot{\WW}(t),
		\\   \nabla\cdot v( t,x) &= 0
		,
		\\
		v( 0,x)&= v_{0} (x)  \Pas
		\commaone x\in\TT^3 \commaone t\in [0,T]
		,
	\end{split}
	\label{EQ35}
\end{align}
where 
\begin{align}
	w=\ukm
	\comma
	v_0 = v^{(k)}_{0}
	\comma
	\zeta=\zeta_{k-1}:=\mathds{1}_{k=0}+\mathds{1}_{k>0}\Pi_{i=0}^{k-1}\psi_i,
      \andand
	\quad
	\psi_i=\psi(v^{(i)})
	\label{EQ36}
\end{align}
are given. 
We shall solve \eqref{EQ35} in two steps,  the first involving the 
analysis of the equation that is linearized on the cut-off level below, which corresponds to the fixed-point iteration scheme for \eqref{EQ35},
\begin{align}
	\begin{split}
		\partial_t v -\Delta v
		&=
		-
		\psi \bar{\psi} \phi \bar{\phi}\Bigl(
		\mathcal{P}\bigl( \bar{w}\cdot \nabla \bar{w}\bigr)
		+
		\zeta
		\mathcal{P}\bigl( w\cdot \nabla \bar{w}\bigr)
		+
	    \zeta
		\mathcal{P}\bigl( \bar{w}\cdot \nabla w\bigr)\Bigr)
		\\&\indeq\indeq
		+
		\psi \bar{\psi} \phi \bar{\phi} \zeta
		(\sigma(t, \bar{w}+w)
		-
		\sigma(t, w)
		)\dot{\WW}(t),
		\\   \nabla\cdot v( t,x) &= 0
		,
		\\
		v( 0,x)&=v_{0} (x)  \Pas
		\commaone x\in\TT^3 \commaone t\in [0,T]
		,
	\end{split}
	\label{EQ37}
\end{align}
where $\phi$, $\psi$, $w$, $\zeta$, and $v_0$ are as in \eqref{EQ33}--\eqref{EQ34} and \eqref{EQ36}, and
\begin{align}
	\bar{\psi}
	:=\psi(\bar{w})
	\andand
	\bar{\phi}
	:=\phi(\bar{w})
	.\label{EQ38}
\end{align}
It is necessary to first establish the existence of iterative
solutions of \eqref{EQ35} before demonstrating the convergence of their solutions. Thus, in Lemma~\ref{LM01}, we prove that \eqref{EQ37} has a unique global solution, ensuring the existence of all iterative solutions. Subsequently, in Lemma~\ref{LM02}, we establish the convergence of these iterative solutions, thereby providing the existence of a solution to~\eqref{EQ35}. In Section~\ref{sec.global}, we derive global energy bounds (Lemmas~\ref{L07} and~\ref{LM04}) and pointwise energy controls (Lemma~\ref{L09}) for the solution of \eqref{EQ35} (and also \eqref{EQ85} with $k$ present). Finally, we demonstrate that the stopping time associated with the cutoff functions in \eqref{EQ33}--\eqref{EQ34} is almost surely positive and finite with small probability.

Next, we state the existence and uniqueness result for~\eqref{EQ37}, recalling the definitions of $\psi$, $\phi$, $\zeta$, $w$, $v_0$, $\bar{\psi}$, and $\bar{\phi}$ given in~\eqref{EQ36} and~\eqref{EQ38}. 
We emphasize that the goal of the next two lemmas is to solve for $v=v^{(k)}$ in \eqref{EQ35} 
given $w=u^{(k-1)}$. Therefore, the condition $w \in L^2(\Omega; C([0,T], H^{\frac{1}{2}+\delta}))\cap
L^2(\Omega; L^2([0,T],H^{\frac{3}{2}+\delta}))$ is our induction assumption.

\cole
\begin{Lemma}
\label{LM01}
Let $T>0$, $k\in\mathbb{N}_0$, and
$w,\bar{w} \in L^2(\Omega; C([0,T], H^{\frac{1}{2}+\delta}))\cap
L^2(\Omega; L^2([0,T],H^{\frac{3}{2}+\delta}))$ be as above. Suppose that $v_0$, $w$, and $\bar{w}$ are divergence and average-free. Then,
the system \eqref{EQ37} has a unique strong solution
$v \in L^2(\Omega; C([0,T], H^{\frac{1}{2}+\delta}))\cap L^2(\Omega; L^2([0,T], H^{\frac{3}{2}+\delta}))$ that is divergence and average-free.
\end{Lemma}
\colb

\begin{proof}[Proof of Lemma~\ref{LM01}]
We employ a fixed-point strategy to solve~\eqref{EQ37} for a fixed~$k$.
Denote
	\begin{align}
		\psi^{(m-1)}
		:=\psi(v^{(m-1)})
              \andand
		\phi^{(m-1)}
		:=\phi(v^{(m-1)}),
		\label{EQ39}
	\end{align}
for $m \in \mathbb{N}_0$, which are, by definition, functions of $t$ only. Consider the iteration scheme
   \begin{align}
   	\begin{split}
   		(\partial_t  -\Delta) v^{(m)}
   		&=
   		-
   		\psi^{(m-1)} \phi^{(m-1)}
   		\bigl(f -g
   		\dot{\WW}(t)
		\bigr),
   		\\   \nabla\cdot v^{(m)}( t,x) &= 0
   		,
   		\\
   		v^{(m)}( 0,x)&= v_{0} (x)  \Pas
	\commaone x\in\TT^3 \commaone t\in [0,T]
   		,
   	\end{split}
   	\label{EQ40}
   \end{align}
where, under the incompressibility condition, 
\begin{align}
	f=\bar{\psi} \bar{\phi}\nabla \cdot \mathcal{P}
	\bigl(
	( \bar{w} \otimes \bar{w})
	+
	\zeta
	( w \otimes \bar{w})
	+
	\zeta
	( \bar{w} \otimes w)
	\bigr)
	\llabel{EQ41}
\end{align}
and
\begin{align}
	g=\bar{\psi}  \bar{\phi} \zeta
	\left( \sigma(t, \bar{w}+w)
	-
	\sigma(t, w)\right).
	\llabel{EQ42}
\end{align}

We now set $v^{(-1)}=0$ and show that all $v^{(m)}$ are well-defined,
   starting with $m=0$.
Since $\theta(0)=1$, the initial value problem~\eqref{EQ40} reduces to
   \begin{align}
	\begin{split}
		(\partial_t  -\Delta) v^{(0)}
		&=
		-
		f +g
		\dot{\WW}(t),
		\\   \nabla\cdot v^{(0)}( t,x) &= 0
		,
		\\
		v^{(0)}( 0,x)&=v_{0} (x)  \Pas
	\commaone x\in\TT^3 \commaone t\in [0,T],
	\end{split}
	\label{EQ43}
\end{align}
when $m=0$. First,
\begin{equation*}
\mathbb{E}\Vert v^{(0)}(0,\cdot)\Vert_{H^{\frac{1}{2}+\delta}}^2
=
\mathbb{E}\Vert v_0\Vert_{H^{\frac{1}{2}+\delta}}^2
< \infty.
\end{equation*}
Furthermore, by the Sobolev estimate~\eqref{EQ23}, 
 \begin{align}
 	 \begin{split}
	\int_{0}^{T}\Vert f\Vert_{H^{{\delta - \frac12}}}^2 \,dt
	&\lec 		
\int_{0}^{T}\bar{\psi}
\Vert\bar{w}\Vert_{H^{\frac{1}{2}+\delta}}^{1+\delta}
\Vert \bar{w}\Vert_{H^{\frac{3}{2}+\delta}}^{1-\delta}
\left(
\Vert\bar{w}\Vert_{H^{\frac{1}{2}+\delta}}^{1+\delta}
\Vert \bar{w}\Vert_{H^{\frac{3}{2}+\delta}}^{1-\delta}
	+
\zeta
	\Vert w\Vert_{H^{\frac{1}{2}+\delta}}^{1+\delta}
	\Vert w\Vert_{H^{\frac{3}{2}+\delta}}^{1-\delta}
	\right) \,dt 
	\\	&\lec _k \int_{0}^{T}\left(
	\bar{\psi}\Vert \bar{w}\Vert_{H^{\frac{3}{2}+\delta}}^{2-2\delta}
	+
	\zeta 
	\Vert w\Vert_{H^{\frac{3}{2}+\delta}}^{2-2\delta}
	\right) \,dt 
	\lec_{k}
	T^{\delta}
	,\Pas
	 \end{split}\label{EQ46}
\end{align}
where in the last step we have used the Cauchy-Schwarz inequality 
and the properties of the cutoff functions.
Lastly, we may bound $g$ using the assumption~\eqref{EQ04} on the noise coefficient as
	 \begin{align}
	 	\int_{0}^{T}\Vert g\Vert_{\mathbb{H}^{\frac{1}{2}+\delta}}^2 \,dt
	 	\lec
	 	 \int_{0}^{T}\bar{\psi}^2\Vert \bar{w}\Vert_{H^{\frac{1}{2}+\delta}}^2 \,dt
	 	  \lec_k T, \Pas
	 	  \label{EQ48}
	 \end{align}
Therefore, invoking Lemma~\ref{L01} componentwise, we conclude that the system~\eqref{EQ43}
has a unique solution $v^{(m)}$ in $L^2(\Omega; C([0,T], H^{\frac{1}{2} +\delta}))\cap L^2(\Omega; L^2([0,T], H^{\frac{3}{2} +\delta}))$
satisfying
	\begin{align}
			\EE\biggl[\sup_{0\leq t\leq T}\Vert v^{(0)}(t,\cdot)\Vert_{H^{\frac{1}{2} +\delta}}^2
			+\int_0^{T} \Vert v^{(0)}(t,\cdot)\Vert_{H^{\frac{3}{2}+\delta}}^2 dt\biggr]
			\le C(1+T^{\delta}+T),
		\label{EQ49}
	\end{align}
where the constant $C$ on the right-hand side of \eqref{EQ49} depends on our hidden index~$k$.

We proceed inductively to justify the existence of each subsequent $v^{(m)}$ and establish its energy bound. Observe that the same bound as in~\eqref{EQ49} applies to all $v^{(m)}$, since $0\leq \psi^{(m-1)},\phi^{(m-1)}\leq 1$ for all $m \in \mathbb{N}_0$. In addition, the divergence-free condition
\eqref{EQ40}$_2$, which holds $\mathbb{P}$-almost surely on $\mathbb{T}^3 \times [0,T]$, is preserved throughout this process. This follows from applying $\div$ to both sides of~\eqref{EQ40}$_1$ and \eqref{EQ40}$_3$
and using that $v_0$, $f$, and $g$ are divergence-free.
Since $\nabla \cdot v^{(m)}$
solves the homogeneous deterministic heat equation
with the zero initial datum, it must be zero. The average-free property is a consequence of the equation's structure and the assumptions on $v_0$, $\sigma$, $w$, and~$\bar{w}$.

Next, we move on to the contraction mapping argument. Clearly, $V^{(m)}:= v^{(m)}-v^{(m-1)}$ is a solution to
 \begin{align}
	\begin{split}
		(\partial_t  -\Delta) V^{(m)}
		&=
		-
		(\Psi^{(m-1)} \phi^{(m-1)} + \psi^{(m-2)} \Phi^{(m-1)})
		(f -g
		\dot{\WW}(t)
		),
		\\   \nabla\cdot V^{(m)}( t,x) &= 0
		,
		\\
		V^{(m)}( 0,x)&=0  \Pas
		\commaone x\in\TT^3 \commaone t\in [0,T]
		,
	\end{split}
	\llabel{EQ50}
\end{align}
where 
	\begin{align*}
	\Psi^{(m-1)}:=\psi^{(m-1)}-\psi^{(m-2)}
	\andand
	\Phi^{(m-1)}:=\phi^{(m-1)}-\phi^{(m-2)}.
\end{align*}
Employing the estimates~\eqref{EQ46}--\eqref{EQ48} and Lemma~\ref{L01} componentwise, we conclude that
	\begin{align}
		\begin{split}
	\EE&\biggl[\Vert V^{(m)}\Vert_{C([0,T_0],H^{\frac{1}{2}+\delta})}^2
	+\Vert V^{(m)}\Vert_{L^2_{T_0} H^{\frac{3}{2}+\delta}}^2\biggr]
	\\&\lec
	 \mathbb{E} \biggl[\int_{0}^{T_0} 
	  \left((\Psi^{(m-1)})^2+(\Phi^{(m-1)})^2\right)
	   (\Vert f\Vert_{H^{{\delta - \frac12}}}^2+\Vert g\Vert_{\mathbb{H}^{\frac{1}{2}+\delta}}^2) \,dt\biggr]
       \\&
	\lec_k
	(T_0+T_0^{\delta})\mathbb{E}\biggl[\Vert V^{(m-1)}\Vert_{C([0,T_0], H^{\frac12 +\delta})}^2
	+
	\Vert V^{(m-1)}\Vert_{L^2_{T_0} H^{\frac{3}{2}+\delta}}^2
	\biggr] 
        ,
	\label{EQ51}
	\end{split}
\end{align}
for all $T_0\in [0,T]$. Setting $T_0>0$ sufficiently small, we have
\begin{align}
	\EE\biggl[\Vert V^{(m)}\Vert_{C([0,T_0], H^{\frac{1}{2}+\delta})}^2
	+\Vert V^{(m)}\Vert_{L^2_{T_0}H^{\frac{3}{2}+\delta}}^{\frac{3}{2}+\delta}\biggr]
	\le
	\frac12
	\EE\biggl[\Vert V^{(m-1)}\Vert_{C([0,T_0], H^{\frac{1}{2}+\delta})}^2
	+\Vert V^{(m-1)}\Vert_{L^2_{T_0}H^{\frac{3}{2}+\delta}}^2\biggr]
	,
	\llabel{EQ52}
\end{align}
which implies that $v^{(m)}$ converges to some function
  \begin{equation*}
   v\in L^2(\Omega; C([0,T_0], H^{\frac{1}{2}+\delta}))\cap L^2(\Omega; L^2([0,T_0], H^{\frac{3}{2}+\delta}))   
   .
  \end{equation*}
Now, we take the limit in the (probabilistically) strong formulation
\begin{align}
	\begin{split}
		(v^{(m)}( t),\varphi)
		=& ( v_0,\varphi)+\int_0^t \bigl((v^{(m)}, \Delta \varphi) 
		   - ( \phi^{(m-1)}\psi^{(m-1)}f, \varphi)\bigr)\,ds
		    \\&+\int_0^t (\phi^{(m-1)}\psi^{(m-1)}g,\varphi)\,d\WW_s
		\comma (t,\omega)\text{-a.e.}
	\end{split}
	\llabel{EQ54}
\end{align}
and the incompressibility identity 
\begin{align}
	\int_0^t (v^{(m)},\nabla \varphi) \,ds = 0
	\comma (t,\omega)\text{-a.e.},
	\llabel{EQ55}
\end{align}
where $\varphi \in C^\infty(\mathbb{T}^3)$ and $t\in[0, T_0]$.
Using the dominated convergence theorem and the exponential rate of convergence,
we conclude that
\begin{align}
	(v^{(m)}(t), \varphi)+ \int_0^t (v^{(m)}, \Delta \varphi)\,ds 
	 \to
	  (v(t), \varphi)+ \int_0^t (v, \Delta \varphi)\,ds 
	   \comma (t,\omega)\text{-a.e.}
	    \llabel{EQ56}
\end{align}
and
\begin{align}
	\int_0^t (v^{(m)},\nabla \varphi) \,ds 
	\to \int_0^t (v,\nabla \varphi) \,ds
	   \comma (t,\omega)\text{-a.e.},
	\llabel{EQ60}
\end{align}
as $m \to \infty$.
For the nonlinear terms, recalling \eqref{EQ36},
we write
\begin{align}
	\mathbb{E}\biggl[\sup_{t}|\psi^{(m)}-\psi|^2\biggr]
	 \lec
	  \mathbb{E}\biggl[
	   \Vert v^{(m)}-v\Vert_{C_{T_0}H^{\frac12 +\delta}\cap L^2_{T_0}H^{\frac{3}{2}+\delta}}^2\biggr]
	   \llabel{EQ57}
\end{align}
and
\begin{align}
	\mathbb{E}\biggl[\sup_{t}|\phi^{(m)}-\phi|^2\biggr]
	\lec
	\EE \biggl[\Vert v^{(m)}-v\Vert_{C_{T_0}H^{\frac12}\cap L^2_{T_0}H^\frac{3}{2}}^2\biggr]
	,\llabel{EQ58}
\end{align}
from where it follows that
\begin{align}
	\psi^{(m)} \to \psi \in L^2(\Omega; L^\infty_{T_0}) 
     \andand
	 \phi^{(m)} \to \phi \in L^2(\Omega; L^\infty_{T_0}).
	 \llabel{EQ59}
\end{align}
Utilizing this, we obtain 
\begin{align}
	\mathbb{E} \biggl[\biggl| 
	  \int_0^{T_0} \left( (\phi^{(m-1)}\psi^{(m-1)}-\phi \psi)f, \varphi\right)\,ds
	   \biggr| \biggr]
	    \lec_k
	     \mathbb{E}\bigl[\sup_{t}\bigl(
	      |\psi^{(m-1)} - \psi|+|\phi^{(m-1)} - \phi|
	                             \bigr)
	               \bigr]
	      \to 0
   ,
	      \llabel{EQ61}
\end{align}
as $m \to \infty$. 
Finally, using the Burkholder-Davis-Gundy inequality, we get
\begin{align}
	\begin{split}
	\mathbb{E}&\biggl[\sup_{t}
 	 \biggl|
	  \int_0^t \left((\phi^{(m-1)}\psi^{(m-1)}-\phi \psi )g,\varphi\right)\,d\WW_s
	  \biggr|  \biggr]
	   \\&
	   \lec
	    \mathbb{E}\biggl[
	      \biggl(
	     \int_0^{T_0} (\phi^{(m-1)}\psi^{(m-1)}-\phi \psi )^2 \Vert g\Vert_{\mathbb{L}^{^2}}^2 \,dt
	      \biggr)^{1/2}\biggr]
         \\&
	      \lec_k
	     \mathbb{E}\biggl[\sup_{t}\biggl(
	     |\psi^{(m-1)} - \psi|+|\phi^{(m-1)} - \phi|\biggr)\biggr]
	     \to 0, 
	     \llabel{EQ62}
\end{split}
\end{align}
as $m \to \infty$.
This establishes the existence of a strong solution 
to \eqref{EQ37} on $[0,T_0]$.
We note that the smallness of $T_0>0$ does not depend on the initial datum.
Therefore, by repeating the steps a finite number of times, we establish the existence of a strong solution on $[0,T]$.

To prove the uniqueness, we assume that $v$ and $\tilde{v}$ are two solutions of~\eqref{EQ37}
on $[0,T]$, with their corresponding cutoffs denoted by $(\psi,\phi)$
and $(\tilde{\psi},\tilde{\phi})$, respectively.
Then, their difference $V=v-\tilde{v}$ solves
\begin{align}
	\begin{split}
		(\partial_t  -\Delta) V
		&
	    =
		-
		(\Psi \phi + \tilde{\psi}\Phi) f
		+
		(\Psi \phi + \tilde{\psi} \Phi) g
		\dot{\WW}(t),
		\\
		\nabla\cdot V (t,x)
		&
		= 0
		,
		\\
		V( 0,x)
		&
		=0  \Pas
		\commaone x\in\TT^3 \commaone t\in [0,T]
		,
	\end{split}
	\llabel{EQ63}
\end{align}
where $\Psi:=\psi-\tilde{\psi}$ and	$\Phi:=\phi-\tilde{\phi}$. Following the approach leading to~\eqref{EQ51}, we claim
\begin{align}
	\EE\Bigl[\Vert V\Vert_{C([0,T_0], H^{\frac{1}{2}+\delta})}^2
	+\Vert V\Vert_{L^2_{T_0}H^{\frac{3}{2}+\delta}}^2
	    \Bigr]
	\le
	\frac{1}{2}
	\EE\Bigl[\Vert V\Vert_{C([0,T_0], H^{\frac{1}{2}+\delta})}^2
	+\Vert V\Vert_{L^2_{T_0}H^{\frac{3}{2}+\delta}}^2\Bigr]
      ,
	\llabel{EQ64}
\end{align}
for a positive but sufficiently small~$T_0$. This establishes the pathwise uniqueness of the solution on $[0,T_0]$, and then on $[0,T]$, by repeating the argument on subsequent subintervals, thus concluding the proof.
\end{proof}

Now, we present the existence and uniqueness result for~\eqref{EQ35}. Recall the notation introduced in~\eqref{EQ33}, \eqref{EQ34}, \eqref{EQ36}, and~\eqref{EQ39}. 

\cole
\begin{Lemma}
\label{LM02}
Let $T>0$, $k\in\mathbb{N}_0$, and
$w \in L^2(\Omega; C([0,T], H^{\frac{1}{2}+\delta}))\cap L^2(\Omega; L^2([0,T],H^{\frac{3}{2}+\delta}))$
be as above. Suppose that $v_0$ and $w$ are divergence and average-free. Then,
the system \eqref{EQ35} has a unique strong solution
$v \in L^2(\Omega; C([0,T], H^{\frac{1}{2}+\delta}))\cap L^2(\Omega; L^2([0,T], H^{\frac{3}{2}+\delta}))$ that is divergence and average-free.
\end{Lemma}
\colb

\begin{proof}[Proof of Lemma~\ref{LM01}]
Consider the iteration for~\eqref{EQ35}, which reads
\begin{align}
	\begin{split}
		(\partial_t -\Delta) v^{(m)}
		=&	f^{(m)} + g^{(m)}\dot{\WW}(t)
                ,
		\\   \nabla\cdot v^{(m)}( t,x) =& 0
		,
		\\
		v^{(m)}( 0,x)=& v_{0} (x)  \Pas
		\commaone x\in\TT^3 \commaone t\in [0,T]
		,
	\end{split}
	\label{EQ65}
\end{align}
where $m \in \mathbb{N}$, 
\begin{align}
	\begin{split}
	f^{(m)}
		=&
		-
		\psi^{(m)} \psi^{(m-1)} \phi^{(m)}\phi^{(m-1)}
		\nabla\cdot \mathcal{P}\bigl( v^{(m-1)} \otimes v^{(m-1)}\bigr)
		\\&-
		\psi^{(m)} \psi^{(m-1)} \phi^{(m)}\phi^{(m-1)}
		\zeta
		\nabla \cdot\mathcal{P}\bigl( w \otimes v^{(m-1)}\bigr)
		\\&-
		\psi^{(m)} \psi^{(m-1)} \phi^{(m)}\phi^{(m-1)}
		\zeta
		\nabla \cdot\mathcal{P}\bigl( v^{(m-1)}\otimes w\bigr)
          ,
	\end{split}
	\llabel{EQ66}
\end{align}
and
\begin{equation}	\llabel{EQ67}
	g^{(m)}=	\psi^{(m)} \psi^{(m-1)} \phi^{(m)}\phi^{(m-1)}
	\zeta\bigl( 
	\sigma(t, v^{(m-1)}+w) - \sigma(t, w)
	\bigr).
\end{equation}
We also set $v^{(0)}$ to be the solution to the incompressible homogeneous
heat equation with the initial datum~$v_0$. By Lemma~\ref{LM01}, the system~\eqref{EQ65} has a unique solution $v^{(m)}$ on $[0,T]$ for every $m \in \mathbb{N}$, and the solution is both, divergence-free and average-free. Moreover, the sequence $\{v^{(m)}\}$
is uniformly bounded
in~$L^2(\Omega; C([0,T], H^{\frac{1}{2}+\delta}))\cap L^2(\Omega;
L^2([0,T], H^{\frac{3}{2}+\delta}))$; this can be seen from their energy estimates
\begin{align}
	\begin{split}
	\EE\Bigl[&\Vert v^{(m)}\Vert_{C([0,T], H^{\frac{1}{2}+\delta})}^2
	+\Vert v^{(m)}\Vert_{L^2_TH^{\frac{3}{2}+\delta}}^2\Bigr]
	\\&\lec
	\EE \Bigl[\Vert v_0\Vert_{H^{\frac{1}{2}+\delta}}^2\Bigr]
	+
	\mathbb{E} \biggl[\int_{0}^{T} 
	(\Vert f^{(m)}\Vert_{H^{\delta - \frac12}}^2+\Vert g^{(m)}\Vert_{\mathbb{H}^{\frac{1}{2}+\delta}}^2) \,dt
	\biggr],
	\label{EQ69}
	\end{split}
\end{align}
as a result of Lemma~\ref{L01}.
Similarly to \eqref{EQ46} and \eqref{EQ48},
we have
\begin{align}
	\int_{0}^{T} 
	\Vert f^{(m)}\Vert_{H^{\delta - \frac12}}^2 \,dt
	   \lec_{k}
	   T^{\delta}
        \andand
	  \int_{0}^{T} 
	   \Vert g^{(m)}\Vert_{\mathbb{H}^{\frac{1}{2}+\delta}}^2 \,dt
	   \lec_{k}
	   T, \Pas
	   \llabel{EQ70}
\end{align}
Therefore, the right-hand side of \eqref{EQ69} is independent of~$m$.

We now show that $\{v^{(m)}\}$ has a limit. Denote
\begin{align}
	V^{(j)}=v^{(j)}-v^{(j-1)}
	\comma
	\Psi^{(j)}=\psi^{(j)}-\psi^{(j-1)}
      \comma \mbox{and }~
	\Phi^{(j)}=\phi^{(j)}-\phi^{(j-1)}
	\comma j\in\{m, m-1\}.
	\llabel{EQ71}
\end{align}
The difference $V^{(m)}$ solves 
\begin{align}
	\begin{split}
		(\partial_t -\Delta) V^{(m)}
		=&
		-F_{\Psi^{(m)}}
		-F_{\Phi^{(m)}}
		-F_{\Psi^{(m-1)}}
		-F_{\Phi^{(m-1)}}
		-F_{V^{(m-1)}}
		\\&+(
		G_{\Psi^{(m)}}
		+G_{\Phi^{(m)}}
		+G_{\Psi^{(m-1)}}
		+G_{\Phi^{(m-1)}}
		+G_{\Sigma^{(m-1)}}
		)\dot{\WW}(t)
		\\   \nabla\cdot V^{(m)}( t,x) =& 0
		,
		\\
		V^{(m)}( 0,x)=&0  \Pas
		\commaone x\in\TT^3  \commaone t\in [0,T]
		,
	\end{split}
	\label{EQ72}
\end{align}
where $F_{\Psi^{(m)}}$ and $F_{\Phi^{(m)}}$ are given by
\begin{align}
	\begin{split}
	F_{\Psi^{(m)}}
	 =&
	  \Psi^{(m)}\psi^{(m-1)} \phi^{(m)}\phi^{(m-1)}
	  \nabla \cdot\mathcal{P}\bigl( v^{(m-1)}\otimes v^{(m-1)}\bigr)
	  \\&+
	  \Psi^{(m)}\psi^{(m-1)} \phi^{(m)}\phi^{(m-1)}\zeta
	  \nabla \cdot\mathcal{P}\bigl( w\otimes v^{(m-1)}\bigr)
	  \\&+
	  \Psi^{(m)}\psi^{(m-1)} \phi^{(m)}\phi^{(m-1)} \zeta
	  \nabla \cdot\mathcal{P}\bigl( v^{(m-1)}\otimes w\bigr)
	  \llabel{EQ73}
	  \end{split}
\end{align}
and 
\begin{align}
	\begin{split}
		F_{\Phi^{(m)}}
		=&
		\Phi^{(m)}(\psi^{(m-1)})^2 \phi^{(m-1)}
		\nabla \cdot\mathcal{P}\bigl( v^{(m-1)}\otimes v^{(m-1)}\bigr)
		\\&+
		\Phi^{(m)}(\psi^{(m-1)})^2 \phi^{(m-1)}\zeta
		\nabla \cdot\mathcal{P}\bigl( w\otimes v^{(m-1)}\bigr)
		\\&+
		\Phi^{(m)}(\psi^{(m-1)})^2 \phi^{(m-1)} \zeta
		\nabla \cdot\mathcal{P}\bigl( v^{(m-1)}\otimes w\bigr)
		,
		\llabel{EQ74}
	\end{split}
\end{align}
the terms $F_{\Psi^{(m-1)}}$ and $F_{\Phi^{(m-1)}}$
read 
\begin{align}
	\begin{split}
		F_{\Psi^{(m-1)}}
		=&
		\Psi^{(m-1)}\psi^{(m-1)} (\phi^{(m-1)})^2
		\nabla \cdot \mathcal{P}\bigl( v^{(m-1)}\otimes v^{(m-1)}\bigr)
		\\&+
		\Psi^{(m-1)}\psi^{(m-1)} (\phi^{(m-1)})^2\zeta
		\nabla \cdot\mathcal{P}\bigl( w\otimes v^{(m-1)}\bigr)
		\\&+
		\Psi^{(m-1)}\psi^{(m-1)} (\phi^{(m-1)})^2 \zeta
		\nabla \cdot\mathcal{P}\bigl( v^{(m-1)}\otimes w\bigr)
		\llabel{EQ75}
	\end{split}
\end{align}
and
\begin{align}
	\begin{split}
		F_{\Phi^{(m-1)}}
		=&
		\Phi^{(m-1)}\psi^{(m-1)}\psi^{(m-2)} \phi^{(m-1)}
		\nabla \cdot\mathcal{P}\bigl( v^{(m-1)}\otimes v^{(m-1)}\bigr)
		\\&+
		\Phi^{(m-1)}\psi^{(m-1)}\psi^{(m-2)} \phi^{(m-1)}\zeta
		\nabla \cdot\mathcal{P}\bigl( w\otimes v^{(m-1)}\bigr)
		\\&+
		\Phi^{(m-1)}\psi^{(m-1)}\psi^{(m-2)} \phi^{(m-1)} \zeta
		\nabla \cdot \mathcal{P}\bigl(v^{(m-1)}\otimes w\bigr)
		,
		\llabel{EQ76}
	\end{split}
\end{align}
while
\begin{align}
	\begin{split}
		F_{V^{(m-1)}}
		=&
		\phi^{(m-1)}\phi^{(m-2)}\psi^{(m-1)}\psi^{(m-2)} 
		\nabla \cdot\mathcal{P}\bigl( V^{(m-1)}\otimes v^{(m-1)}\bigr)
		\\&+
		\phi^{(m-1)}\phi^{(m-2)}\psi^{(m-1)}\psi^{(m-2)} 
		\nabla \cdot\mathcal{P}\bigl( v^{(m-2)}\otimes V^{(m-1)}\bigr)
		\\&+		
		\phi^{(m-1)}\phi^{(m-2)}\psi^{(m-1)}\psi^{(m-2)} 
		\zeta
		 \nabla \cdot\mathcal{P}\bigl( w\otimes V^{(m-1)}\bigr)
		\\&+
		\phi^{(m-1)}\phi^{(m-2)}\psi^{(m-1)}\psi^{(m-2)} 
		\zeta
		\nabla \cdot\mathcal{P}\bigl( V^{(m-1)}\otimes w\bigr)
		;
		\llabel{EQ77}
	\end{split}
\end{align}
also, the five noise factors in~\eqref{EQ72}
are defined by
\begin{align}
	\begin{split}
		G_{\Psi^{(m)}}
		&=
		\Psi^{(m)}\psi^{(m-1)} \phi^{(m)}\phi^{(m-1)}\zeta
		( 
		\sigma(t, v^{(m-1)}+w) - \sigma(t, w)
		)
	,	\\
		G_{\Phi^{(m)}}
		&=
		\Phi^{(m)}(\psi^{(m-1)})^2 \phi^{(m-1)}\zeta
		( 
		\sigma(t, v^{(m-1)}+w) - \sigma(t, w)
		)
	,	\\
		G_{\Psi^{(m-1)}}
		&=
		\Psi^{(m-1)}\psi^{(m-1)} (\phi^{(m-1)})^2\zeta
		( 
		\sigma(t, v^{(m-1)}+w) - \sigma(t, w)
		)
     ,           \\
		G_{\Phi^{(m-1)}}
		&=
		\Phi^{(m-1)}\psi^{(m-1)} \psi^{(m-2)}\phi^{(m-1)}\zeta
		( 
		\sigma(t, v^{(m-1)}+w) - \sigma(t, w)
		)
	,	\\
		G_{\Sigma^{(m-1)}}
		&=
		\psi^{(m-1)}\psi^{(m-2)} \phi^{(m-1)}\phi^{(m-2)}\zeta
		( 
		\sigma(t, v^{(m-1)}+w) - \sigma(t, v^{(m-2)}+w)
		)
		.\llabel{EQ79}
	\end{split}
\end{align}
Following the same approach as in \eqref{EQ46}, \eqref{EQ48}, and \eqref{EQ51},
we conclude that, for all $T_0\in[0, T]$, 
\begin{align}
	\begin{split}
	\EE \biggl[&\int_0^{T_0}
	 (\Vert F_{\Psi^{(j)}}\Vert_{H^{\delta - \frac12}}^2
	  +\Vert F_{\Phi^{(j)}}\Vert_{H^{\delta - \frac12}}^2
	   +\Vert G_{\Phi^{(j)}}\Vert_{\mathbb{H}^{\frac{1}{2}+\delta}}^2
	    +\Vert G_{\Psi^{(j)}}\Vert_{\mathbb{H}^{\frac{1}{2}+\delta}}^2)\,ds\biggr]
	  \\& \lec_k
	   (T_0+T_0^{\delta})\mathbb{E}\Bigl[\Vert V^{(j)}\Vert_{C([0,T_0], H^{\frac12 +\delta})}^2
	   +
	   \Vert V^{(j)}\Vert_{L^2_{T_0} H^{\frac{3}{2}+\delta}}^2
	   \Bigr] 
	   \comma j\in\{m, m-1\}
	   .\label{EQ80}
	   \end{split}
\end{align}
In addition, note that
\begin{align}
	\int_{0}^{T_0}\Vert 	G_{\Sigma^{(m-1)}}\Vert_{\mathbb{H}^{\frac{1}{2}+\delta}}^2 \,dt
	\lec
	\int_{0}^{T_0}(\psi^{(m-1)}\psi^{(m-2)} )^2\Vert V^{(m-1)}\Vert_{H^{\frac{1}{2}+\delta}}^2 \,dt
	\lec_k T_0 \Vert V^{(m-1)}\Vert_{C([0,T_0], H^{\frac12 +\delta})}^2
	,
   \label{EQ44}
\end{align}
$\PP$-almost surely,
and, by the Sobolev estimate~\eqref{EQ23}, 
 \begin{align}
	\begin{split}
		&\quad\enskip  \int_{0}^{T_0}\Vert F_{V^{(m-1)}}\Vert_{H^{{\delta - \frac12}}}^2 \,dt
		\\ &\lec 		
	 \int_{0}^{T_0}\psi^{(m-1)}\psi^{(m-2)} 
			\Vert V^{(m-1)}\Vert_{H^{\frac{1}{2}+\delta}}^{2+2\delta}
		\left(
		\Vert v^{(m-1)}\Vert_{H^{\frac{3}{2}+\delta}}^{2-2\delta}
		+
		\Vert v^{(m-2)}\Vert_{H^{\frac{3}{2}+\delta}}^{2-2\delta}
		+
		\zeta
		\Vert w\Vert_{H^{\frac{3}{2}+\delta}}^{2-2\delta}
		\right) \,dt 
		\\	&\quad +
		 \int_{0}^{T_0}\psi^{(m-1)}\psi^{(m-2)} 
			\Vert V^{(m-1)}\Vert_{H^{\frac{3}{2}+\delta}}^{2-2\delta}
		\left(
		\Vert v^{(m-1)}\Vert_{H^{\frac{1}{2}+\delta}}^{2+2\delta}
		+
		\Vert v^{(m-2)}\Vert_{H^{\frac{1}{2}+\delta}}^{2+2\delta}
		+
		\zeta
		\Vert w\Vert_{H^{\frac{1}{2}+\delta}}^{2+2\delta}
		\right) \,dt 
			\\	&
			\lec_{k}
		T_0^{\delta} \Bigl(
		\Vert V^{(m-1)}\Vert_{C([0,T_0], H^{\frac12 +\delta})}^2
		+
	\Vert V^{(m-1)}\Vert_{L^2_{T_0} H^{\frac{3}{2}+\delta}}^2
               		\Bigr)
		,\Pas
	\end{split}\label{EQ68}
\end{align}
Together, \eqref{EQ44} and \eqref{EQ68} imply
\begin{align}
	\EE \biggl[\int_0^{T_0}
	(\Vert F_{V^{(m-1)}}\Vert_{H^{\delta - \frac12}}^2
	+\Vert G_{\Sigma^{(m-1)}}\Vert_{\mathbb{H}^{\frac{1}{2}+\delta}}^2)\,ds\biggr]
	\lec_k
	(T_0+T_0^{\delta})\EE \Bigl[\Vert V^{(m-1)}\Vert_{C_{T_0}H^{\frac{1}{2}+\delta} \cap L^2_{T_0}H^{\frac{3}{2}+\delta}}^2\Bigr]
	.\label{EQ81}
\end{align}
Combining \eqref{EQ80} and \eqref{EQ81}, we then obtain
\begin{align}
	\begin{split}
&	\bigl(1-C_k(T_0+T_0^{\delta})\bigr)\EE\Bigl[\Vert V^{(m)}\Vert_{C_{T_0}H^{\frac{1}{2}+\delta}}^2
	+\Vert V^{(m)}\Vert_{L^2_{T_0}H^{\frac{3}{2}+\delta}}^2\Bigr]
	\\&\qquad \le
	C_k(T_0+T_0^{\delta})
	\EE\Bigl[\Vert V^{(m-1)}\Vert_{C_{T_0}H^{\frac{1}{2}+\delta}}^2
	+\Vert V^{(m-1)}\Vert_{L^2_{T_0}H^{\frac{3}{2}+\delta}}^2\Bigr]
	,
	\llabel{EQ82}
	\end{split}
\end{align} 
which demonstrates that $\{ v^{(m)}\}$ converges to a function
  \begin{equation}
   v\in L^2(\Omega; C([0,T_0], H^{\frac{1}{2}+\delta}))\cap L^2(\Omega; L^2([0,T_0], H^{\frac{3}{2}+\delta}))   
   \llabel{EQ53}
  \end{equation}
at an exponential rate when $T_0$ is sufficiently small. With this exponential rate of convergence and the uniform energy bound, we can now take the limit in the identities
\begin{align}
	\begin{split}
		(v^{(m)}( t),\varphi)
		=& ( v_0,\varphi)+\int_0^t (v^{(m)}, \Delta \varphi)\,ds
		\\&
		-\int_0^t 
		 ( \psi^{(m)} \psi^{(m-1)} \phi^{(m)}\phi^{(m-1)}
		\mathcal{P}\bigl( v^{(m-1)} \otimes v^{(m-1)}\bigr), \nabla\varphi)\,ds
		\\&
		-\int_0^t 
		(
		\psi^{(m)} \psi^{(m-1)} \phi^{(m)}\phi^{(m-1)}
		\zeta
		\mathcal{P}\bigl( w \otimes v^{(m-1)}\bigr)
		, \nabla\varphi)\,ds
		\\&
		-\int_0^t 
		( \psi^{(m)} \psi^{(m-1)} \phi^{(m)}\phi^{(m-1)}
		\zeta
		\mathcal{P}\bigl( v^{(m-1)}\otimes w\bigr)
		, \nabla\varphi)\,ds
		\\&
		+\int_0^t (\psi^{(m)} \psi^{(m-1)} \phi^{(m)}\phi^{(m-1)}
		\zeta\bigl( 
		\sigma(s,v^{(m-1)}+w) - \sigma(s, w)
		\bigr)
		,\varphi)\,d\WW_s
    \end{split}
   \llabel{EQ47}          
    \end{align}
and
   \begin{align}
   \begin{split}
      \int_0^t (v^{(m)},\nabla \varphi) &\,ds = 0
		\comma (t,\omega)\text{-a.e.} \comma \varphi \in C^\infty(\mathbb{T}^3),
	\end{split}
	\llabel{EQ83}
\end{align}
and establish the existence of a (probabilistically) strong solution to \eqref{EQ35}
on $[0,T_0]$. Since the smallness of $T_0$ does not depend on the initial data, we can further extend this solution to one on $[0,T]$ by repeating the argument in finitely many steps. This part of the proof, as well as the proof of pathwise uniqueness, is nearly identical to the corresponding parts in Lemma~\ref{LM01}, and so we omit it here. This concludes the proof of Lemma~\ref{LM02}.
\end{proof}

\subsection{The global energy control and pointwise estimates}\label{sec.global}
Now, we bring back the index $k$ and rewrite \eqref{EQ35} in terms of
$v^{(k)}$, $u^{(k)}$,
and $u^{(k-1)}$,
\begin{align}
	\begin{split}
		\partial_t v^{(k)}  -\Delta v^{(k)}
		=&
		-
		(\psi_k \phi_k)^2
		\nabla \cdot \mathcal{P}
		\bigl( v^{(k)}\otimes v^{(k)}
		+
		 \zeta_{k-1}u^{(k-1)}\otimes v^{(k)}
		+
		 \zeta_{k-1}v^{(k)}\otimes u^{(k-1)}
		\bigr)\\&
		+
		(\psi_k \phi_k)^2 \zeta_{k-1}\bigl( 
		\sigma(t,v^{(k)}+u^{(k-1)}) - \sigma(t, u^{(k-1)})
		\bigr)
		\dot{\WW}(t),
		\\   \nabla\cdot v^{(k)}( t,x) =& 0
		,
		\\
		v^{(k)}( 0,x)=&v^{(k)}_{0} (x),  \Pas\commaone x\in\TT^3 \commaone t\in [0,T]
		.
	\end{split}
	\label{EQ85}
\end{align}
Lemma~\ref{LM02} states that, given $k$, $u^{(k-1)}$, and $v_0^{(k)}$, the initial value problem \eqref{EQ85} has a unique solution $v^{(k)}$ that exists globally on $[0,T]$. Recalling that $u^{(-1)}=0$, the proofs we presented in Lemma~\ref{LM01} and~\ref{LM02} work for the base step of the induction, i.e., when $k=0$. We provide estimates for this solution that are independent of both time and~$k$, starting with the $H^{\frac{1}{2}}$-estimate. Recall the definitions of the cutoff functions~\eqref{EQ33}--\eqref{EQ34} and the assumptions regarding the noise coefficient~\eqref{EQ04}. 

\cole
\begin{Lemma}[An $H^{\frac{1}{2}}$-energy control]
\label{L07}
Let $\epsilon>0$ and $k\in\mathbb{N}_0$. Suppose that $\vk_0$ is an initial datum of~\eqref{EQ85} satisfying the assertions in Lemma~\ref{L05}; additionally, assume that 
\begin{align}
	\sup_{\Omega\times [0,\infty)} \biggl[\Vert \ukm(t)\Vert_{H^{\frac12}}
	 +\biggl(\int_0^t \Vert \ukm\Vert_{H^{\frac{3}{2}}}^2\biggr)^\frac{1}{2}\,ds\biggr]
	  \leq\epsilon.
	   \label{EQ86}
\end{align}
If $\epsilon$, $\bar{\epsilon}$, and $\epsilon_{\sigma}$ are sufficiently small, 
then there exists a positive constant $C$,
independent of $k$, so that
\begin{align}
\EE \biggl[\sup_{t\in[0,\infty)} \Vert \vk(t)\Vert_{H^\frac12}^2
+\int_0^{\infty} \Vert \vk(t)\Vert_{H^{\frac{3}{2}}}^2 \,dt
\biggr]
\leq C\EE\Bigl[\Vert \vk_0\Vert_{H^\frac12}^2\Bigr].
\label{EQ87}
\end{align}
\end{Lemma}
\colb

\begin{proof}[Proof of Lemma~\ref{L07}]
With $T>0$ fixed,
we invoke Lemma~\ref{L01} for \eqref{EQ85} componentwise, obtaining
 \begin{align}
	\begin{split}
		&\EE\biggl[
		\sup_{0\leq t\leq T}
		\Vert v^{(k)}(t,\cdot)\Vert_{H^{\frac12}}^2
		- \Vert v_0^{(k)}\Vert_{H^{\frac12}}^2
		+ \int_0^{T} \Vert v^{(k)}\Vert_{H^{\frac{3}{2}}}^2 \,dt\biggr]
		\\&\indeq
		\lec 
		\EE\biggl[
		\int_0^{T} (\psi_k \phi_k)^2 \Vert v^{(k)}\otimes v^{(k)}\Vert_{H^{\frac12}}^2
		\,dt\biggr]
		+\EE \biggl[\int_0^{T} (\psi_k \phi_k)^2\zeta_{k-1}\Vert v^{(k)}\otimes u^{(k-1)} \Vert_{H^\frac12}^2
		\,dt\biggr]
		\\&\indeq\indeq
		+
		\EE \biggl[\int_0^{T}\Vert (\psi_k \phi_k)^2 \zeta_{k-1}( 
		               \sigma(t,v^{(k)}+u^{(k-1)}) - \sigma(t, u^{(k-1)})
		                                                                )\Vert_{\mathbb{H}^\frac12}^2\,dt\biggr]
		 = I_1 + I_2 + I_3
		.
	\end{split}
	\label{EQ88}
\end{align}
As an immediate consequence of~\eqref{EQ24} and \eqref{EQ34}, we have 
\begin{align}
	I_1
	 \lec
	  \bar{\epsilon}
	    \EE\biggl[
	    \int_0^{T} \Vert v^{(k)}\Vert_{H^{\frac{3}{2}}}^2
	    \,dt\biggr]
	  .\label{EQ89}
	\end{align}
For $I_2$, we apply the Sobolev estimate~\eqref{EQ23} with $\delta = 0$ and write
\begin{align}
	\begin{split}
	I_2
	 &\lec 
	  \EE\biggl[
	  \int_0^{T} (\psi_k \phi_k)^2\zeta_{k-1}
	  \Vert v^{(k)}\Vert_{H^{\frac{1}{2}}}\Vert v^{(k)}\Vert_{H^{\frac{3}{2}}}
	              \Vert u^{(k-1)}\Vert_{H^{\frac{1}{2}}}\Vert u^{(k-1)}\Vert_{H^{\frac{3}{2}}}
	  \,dt\biggr]
	  \\& \lec
	    \EE\biggl[
	   \int_0^{T} (\phi_k)^2\Vert v^{(k)}\Vert_{H^{\frac{1}{2}}}^2 \Vert u^{(k-1)}\Vert_{H^{\frac{3}{2}}}^2
	   \,dt\biggr]
	   +
	    \EE\biggl[
	   \int_0^{T} \Vert u^{(k-1)}\Vert_{H^{\frac{1}{2}}}^2 \Vert v^{(k)}\Vert_{H^{\frac{3}{2}}}^2
	   \,dt\biggr] = I_{21}+I_{22}
	   .\label{EQ90}
	   \end{split} 
\end{align}
By the smallness assumption~\eqref{EQ86},
\begin{align}
	I_{21}
	 \lec 
	  \EE\biggl[\sup_t \Vert v^{(k)}\Vert_{H^{\frac{1}{2}}}^2 
	  \int_0^{T} \Vert u^{(k-1)}\Vert_{H^{\frac{3}{2}}}^2
	  \,dt\biggr]
	  \lec \epsilon^2
	   \EE\biggl[\sup_t \Vert v^{(k)}\Vert_{H^{\frac{1}{2}}}^2\biggr]
	   \label{EQ91}
\end{align}
and 
\begin{align}
	I_{22}
	 \lec 
	  \epsilon^2 \EE\biggl[
	  \int_0^{T} \Vert v^{(k)}\Vert_{H^{\frac{3}{2}}}^2
	  \,dt\biggr]
	  .\label{EQ92}
\end{align}
From~\eqref{EQ04}, we derive 
\begin{align}
	I_3
	 \lec
	   \epsilon_{\sigma}^2
	   \EE\biggl[
	   \int_0^{T} \Vert v^{(k)}\Vert_{H^{\frac{1}{2}}}^2
	   \,dt\biggr]
	 \lec
	  \epsilon_{\sigma}^2
	 \EE\biggl[
	 \int_0^{T} \Vert v^{(k)}\Vert_{H^{\frac{3}{2}}}^2
	 \,dt\biggr]
	   .\label{EQ93}
\end{align}
Finally, we collect \eqref{EQ88}--\eqref{EQ93}, assuming that 
$\epsilon$, $\bar{\epsilon}$, and $\epsilon_{\sigma}$ are
sufficiently small. Since the required smallness does not depend on $k$ or $T$,
we arrive at~\eqref{EQ87}.
\end{proof}  

Additionally, the $H^{\frac12 +\delta}$ bounds are uniform in both~$k$ and~$T$. 

\cole
\begin{Lemma}[An $H^{\frac{1}{2}+\delta}$ energy control]
	\label{LM04}
	Let $\epsilon>0$ and $k\in\mathbb{N}_0$. Suppose that $\vk_0$ is an initial datum of~\eqref{EQ85} satisfying the assertions in Lemma~\ref{L05} and $u^{(k-1)}$ satisfies~\eqref{EQ86}.
	Then, there exists a positive constant $C$, independent of $k$, such that 
	\begin{align}
		\EE \biggl[\sup_{t\in[0,\infty)} \Vert \vk(t)\Vert_{H^{\frac{1}{2}+\delta}}^2
		+\int_0^{\infty} \Vert \vk(t)\Vert_{H^{\frac{3}{2}+\delta}}^2 \,dt
		\biggr]
		\leq C\EE\Bigl[\Vert \vk_0\Vert_{H^{\frac{1}{2}+\delta}}^2\Bigr]
		\label{EQ94}
	\end{align}
if $\epsilon$, $\bar{\epsilon}$, and $\epsilon_{\sigma}$ are sufficiently small.
\end{Lemma}
\colb

\begin{proof}[proof of Lemma~\ref{LM04}]
Let $T>0$ be fixed. We apply Lemma~\ref{L01} to \eqref{EQ85} componentwise, obtaining
	 \begin{align}
		\begin{split}
			&\EE\biggl[
			\sup_{0\leq t\leq T}
			\Vert v^{(k)}(t,\cdot)\Vert_{H^{\frac{1}{2}+\delta}}^2
			- \Vert v_0^{(k)}\Vert_{H^{\frac{1}{2}+\delta}}^2
			+ \int_0^{T} \Vert v^{(k)}\Vert_{H^{\frac{3}{2}+\delta}}^2 \,dt
			\biggr]
			\\&\indeq
			\lec
			\EE\biggl[
			\int_0^{T} (\psi_k \phi_k)^2 \Vert v^{(k)}\otimes v^{(k)}\Vert_{H^{\frac12  +\delta}}^2
			\,dt\biggr]
			+\EE \biggl[\int_0^{T} (\psi_k \phi_k)^2\zeta_{k-1}\Vert v^{(k)}\otimes u^{(k-1)} \Vert_{H^{\frac12 + \delta}}^2
			\,dt\biggr]
			\\&\indeq\indeq
			+
			\EE \biggl[
			 \int_0^{T}\bigl\Vert (\psi_k \phi_k)^2 \zeta_{k-1}
			\bigl( 
			\sigma(t,v^{(k)}+u^{(k-1)}) - \sigma(t, u^{(k-1)})
			\bigr)
			                    \bigr\Vert_{\mathbb{H}^{\frac12 +\delta}}^2\,dt\biggr]
			= I_1 + I_2 + I_3
			.
		\end{split}
		\label{EQ95}
	\end{align}
First, the estimate~\eqref{EQ25} and the definition~\eqref{EQ34} yield
   \begin{align}
   	I_1
   	 \lec
   	   \bar{\epsilon}^2
   	   \EE \biggl[\int_0^{T} \Vert v^{(k)}\Vert_{H^{\frac{3}{2} + \delta}}^2\,dt\biggr]
   	   ,
	   \llabel{EQ96}
   \end{align}
while for $I_3$ we have
   \begin{align}
   	 I_3
   	  \lec 
   	   \epsilon_{\sigma}^2 \EE\biggl[
   	   \int_0^{T} \Vert v^{(k)}\Vert_{H^{\frac{1}{2}+\delta}}^2
   	   \,dt\biggr]
   	     \lec
   	     \epsilon_{\sigma}^2
   	    \EE\biggl[
   	    \int_0^{T} \Vert v^{(k)}\Vert_{H^{\frac{3}{2}+\delta}}^2
   	    \,dt\biggr]
   	    ,\llabel{EQ97}
   \end{align}
using~\eqref{EQ04}. Finally, recalling the Sobolev inequality~\eqref{EQ23} and the Sobolev embedding, we get
\begin{align}
	\begin{split}
	I_2
	 &\lec
	 \EE \biggl[\int_0^{T}  \phi_k \psi_k
	  \Vert v^{(k)}\Vert_{H^{\frac{1}{2}+\delta}}^{1+\delta}
	 \Vert v^{(k)}\Vert_{H^{\frac{3}{2}+\delta}}^{1-\delta}
	 \Vert u^{(k-1)}\Vert_{H^{\frac{1}{2}+\delta}}^{1+\delta}
	 \Vert u^{(k-1)}\Vert_{H^{\frac{3}{2}+\delta}}^{1-\delta}
	 \,dt\biggr]
	 \\ &\lec
	 \EE \biggl[\int_0^{T}  \phi_k \psi_k \bigl(
	 \Vert v^{(k)}\Vert_{H^{\frac{1}{2}+\delta}}^{2}
	 \Vert u^{(k-1)}\Vert_{H^{\frac{3}{2}+\delta}}^{2}
	 +
	 \Vert v^{(k)}\Vert_{H^{\frac{3}{2}+\delta}}^{2}
	 \Vert u^{(k-1)}\Vert_{H^{\frac{1}{2}+\delta}}^{2}
	 \bigr) \,dt\biggr]
	  \\ &\lec
	  \epsilon^2\left(
	   \EE \biggl[\sup_t \Vert v^{(k)}\Vert_{H^{\frac12 +\delta}}^2\biggr]
	   +\EE \biggl[\int_0^{T} \Vert v^{(k)}\Vert_{H^{\frac{3}{2} + \delta}}^2\,dt\biggr]
	   \right)
	   .
	  \end{split}\label{EQ99}
\end{align}  
	By collecting \eqref{EQ95}--\eqref{EQ99} and choosing
	$\epsilon$, $\bar{\epsilon}$, and $\epsilon_{\sigma}$
	sufficiently small, we achieve~\eqref{EQ94}.
\end{proof}

\begin{Remark}\label{R04}
We note that interpolating
	$\mathbb{H}^{\frac{1}{2}+\delta}$ between
	$\mathbb{H}^{\frac{1}{2}}$ and $\mathbb{H}^{\frac{1}{2}+2\delta}$
	allows us to obtain 
	\begin{align}
		I_3
		\le 
		C\epsilon_{\sigma} \EE\biggl[
		\int_0^{T} \phi_k\Vert v^{(k)}\Vert_{H^{\frac{1}{2}}}\Vert v^{(k)}\Vert_{H^{\frac{1}{2}+2\delta}}
		\,dt\biggr]
		\le
		C \epsilon_{\sigma}
		\EE\biggl[
		\int_0^{T} \Vert v^{(k)}\Vert_{H^{\frac{3}{2}+\delta}}^2
		\,dt\biggr],
		\llabel{EQ097}
	\end{align}
under the assumptions
   \begin{align}
   	\begin{split}
   			 \Vert\sigma(t, u_1)-\sigma(t, u_2)\Vert_{\mathbb{H}^{\frac12 }} 
   		&\le \epsilon_{\sigma}\Vert  u_1-u_2\Vert_{H^{\frac12}},\\
   	\Vert \sigma(t,u_1)-\sigma(t,u_2)\Vert_{\mathbb{H}^{\frac{1}{2}+2\delta}}
   	 &\le C \Vert u_1-u_2\Vert_{H^{\frac{1}{2}+2\delta}}
   	 .
   	\end{split}\label{EQ098}
   \end{align}
   Therefore, we may replace \eqref{EQ04} by \eqref{EQ098}
   which only requires smallness in~$H^\frac12$.
\end{Remark}

\begin{Remark}\label{R01}  
When we do not assume that $\epsilon_\sigma$ is small in Lemmas~\ref{L07} and \ref{LM04}, we obtain
   \begin{align}
\EE\biggl[\sup_{0\leq t\leq T}\Vert\vk(t,\cdot)\Vert_{H^{\frac12 +\alpha}}^2
+\int_0^{T} \Vert\vk(t,\cdot)\Vert_{H^{\frac{3}{2} +\alpha}}^2 \,dt\biggr]
\leq C
\EE\biggl[
\Vert\vk_{0}\Vert_{H^{\frac12 +\alpha}}^2
+
\int_0^{T} 
\Vert \vk\Vert_{H^{\frac12 +\alpha}}^2
\,dt
\biggr],\llabel{EQ101}
\end{align}
from which, by employing Gr\"onwall's lemma, we derive
   \begin{align}
\begin{split}
&\EE\biggl[\sup_{0\leq t\leq T}\Vert\vk(t,\cdot)\Vert_{H^{\frac12 +\alpha}}^2
+\int_0^{T}\Vert\vk(t,\cdot)\Vert_{H^{\frac{3}{2} +\alpha}}^2 \,dt\biggr]
\leq C_T
\EE\bigl[
\Vert\vk_{0}\Vert_{H^{\frac12 +\alpha}}^2
\bigr]
,
\end{split}
\llabel{EQ100}
\end{align}
for any fixed $T$ and $\alpha \in\{0, \delta\}$.
Following the steps in the above proof, we then obtain the assertion of Lemma~\ref{LM04} on $[0,T]$
without utilizing that $\epsilon_\sigma$ is small.
\end{Remark}  

Given the existence of $v^{(k)}$ for all times $t$, we can establish a pointwise control of the $H^\frac{1}{2}$-norm of~$v^{(k)}$, which ensures the assumption~\eqref{EQ86} in Lemma~\ref{L07}; see~\eqref{EQ31}.
 
\cole
\begin{Lemma}[A pointwise $H^\frac{1}{2}$ control]
\label{L09}
  Let $k\in\mathbb{N}_0$, and consider 
  a sequence of initial data $\{\vk_0\}_{k\in\NNp}$ 
  satisfying \eqref{EQ26}--\eqref{EQ29}. 
  Then, for the unique solution $v^{(k)}$ of the system \eqref{EQ85}, we have
 \begin{align}
   \sup_{\Omega \times [0,\infty)}
    \biggl(\Vert \vk(t)\Vert_{H^\frac{1}{2}} + \left(\int_0^t \Vert \vk(s)\Vert_{H^\frac{3}{2}}^2 \,ds\right)^\frac12\biggr)
   \leq
   \frac{\bar{\epsilon}}{2^{k-1}}
   ,\label{EQ102}
  \end{align}
where $\bar{\epsilon} \in (2\epsilon_0,1)$.
\end{Lemma}
\colb

\begin{proof}[proof of Lemma~\ref{L09}]
The argument for this lemma follows from the proof of~\cite[Lemma 3.5]{KX2}
upon replacing $\Vert \vk(t)\Vert_{L^{3}}$ with $Q_{k,0}(t)$, where
\begin{align}
	{Q}_{k,\alpha}(t)
	=\Vert \vk(t)\Vert_{H^{\frac{1}{2}+\alpha}} + \left(\int_0^t \Vert \vk(s)\Vert_{H^{\frac{3}{2}+\alpha}}^2 \,ds\right)^\frac12
	\comma \alpha=0, \delta.\llabel{EQ104}
\end{align}
The argument relies on the fact that the energy of $\vk$ decays once the energy exceeds the $\bar{\epsilon}/2^{k-1}$-threshold imposed by the cutoff functions; see~\eqref{EQ34}.
\end{proof}

Denote by $\tau_k$ and $\rho_k$ the stopping times when $Q_{k,0}(t)$
first reaches $\fractext{\bar{\epsilon}}{2^{k}}$ and $Q_{k,\delta}(t)$
reaches $M_k$, respectively. Then define
  \begin{align}
   \tau^{k}
   = \tau_{1} \wedge \rho_{1} \wedge \cdots \wedge \tau_{k} \wedge \rho_{k}
    \comma k\in\mathbb{N}_0
   \label{EQ105}
  \end{align}
and 
  \begin{align}
   \tau=\inf_k \tau^k=\lim_k \tau^k
   .
   \label{EQ106}
  \end{align}
Consequently, 
  \begin{align}
  \sup_{t\in [0,\tau)}Q_{k,0}(t)
  \leq
  \frac{\bar{\epsilon}}{2^{k}}
  \comma k\in\NNz
  \llabel{EQ107}
  \end{align}
and 
  \begin{align}
  \sup_{t\in [0,\tau)}Q_{k,\delta}(t)
  \leq
  M_k
  \comma k\in\mathbb{N}_0
  .
  \llabel{EQ108}
  \end{align}
Now, we show that $\tau$ is finite with only a small probability.

\cole
\begin{Lemma}
\label{L10}
Let $\bar \epsilon$ and $\epsilon_{\sigma}$ be sufficiently small, and let $\{\vk_0\}_{k\in\NNp}$ be a sequence of initial data for~\eqref{EQ85} satisfying the assertions of Lemma~\ref{L05}. 
For every $p_0\in(0,1)$, we have
\[
\PP(\tau<\infty)\leq p_0,
\]
provided
$\epsilon_0$ is sufficiently small relative to $\bar{\epsilon}$ and $M_k$ is sufficiently large relative to $\MM_k$; see~\eqref{EQ11}, \eqref{EQ30}, \eqref{EQ33}, and~\eqref{EQ34} for their definitions.
\end{Lemma}
\colb

\begin{proof}[Proof of Lemma~\ref{L10}]
By Lemmas~\ref{L07} and~\ref{LM04},
\begin{align}
\begin{split}
\EE\biggl[\sup_{0\leq t< \infty}Q^2_{k,\alpha}(t)\biggr]
\leq 
C\EE\Bigl[\Vert \vk_0\Vert_{H^{\frac{1}{2}+\alpha}}^2\Bigr]
\comma \alpha=0,\delta,
\llabel{EQ109}
\end{split}
\end{align}
where $C>0$ does not depend on~$k$.
Utilizing Markov's inequality, we obtain
\begin{align}
  \begin{split}
  \PP\biggl(\sup_{0\leq t< \infty}Q_{k,0}(t)\geq \frac{\bar{\epsilon}}{2^{k}}\biggr)
  \leq 
  \frac{C\EE\Bigl[\Vert \vk_0\Vert_{H^\frac12}^2\Bigr]}{\bar{\epsilon}^2/2^{2k}}
  \leq 
  \frac{C\epsilon_0^2/4^{2k}}{\bar{\epsilon}^2/2^{2k}}
  \leq 
  \frac{C}{2^{2k}}
  \left(
   \frac{\epsilon_0}{\bar\epsilon}
  \right)^{2}
  \leq
  \frac{p_0}{2^{2k+2}}
  \comma 
  k\in\mathbb{N}_0
  \end{split}
   \label{EQ110}
  \end{align}
if $\epsilon_0$ is sufficiently small compared to~$\bar{\epsilon}$.
Moreover, we may choose $M_k$ sufficiently large with respect to
$\MM_k$ and write
\begin{align}
\begin{split}
\PP\biggl(\sup_{0\leq t< \infty}Q_{k,\delta}(t)\geq M_k\biggr)
\leq 
\frac{C\EE[\Vert \vk_0\Vert_{H^{\frac{1}{2}+\delta}}^2]}{M_k^2}
\leq 
\frac{C\MM_k^2}{M_k^2}
\leq 
\frac{p_0}{2^{2k+2}}
\comma 
k\in\mathbb{N}_0
.
\end{split}\label{EQ111}
\end{align}
Combining \eqref{EQ110} and \eqref{EQ111}, we arrive at
\begin{align}
\mathbb{P}(\tau_k<\infty)
\leq
\frac{p_0}{2^{2k+2}}
 \andand
\mathbb{P}(\rho_k<\infty)
\leq
\frac{p_0}{2^{2k+2}}
\comma k\in \mathbb{N}_0.
   \llabel{EQ112}
\end{align} 
Therefore, we obtain
  \begin{align}
  \begin{split}
   \PP(\tau^k<\infty)
   &\leq
   \sum_{j=0}^{k}
   \mathbb{P}
   \Bigl(
   \sup_{0\leq t< \infty}Q_{k,\delta}(t)\geq M_j
   \Bigr)
   +
   \sum_{j=0}^{k}
   \mathbb{P}
   \Bigl(
   \sup_{0\leq t< \infty}Q_{k,0}(t)\geq \frac{\bar{\epsilon}}{2^{j}}
   \Bigr)
   \leq p_0
    \comma k\in\mathbb{N}_0
   ,
  \end{split}
   \llabel{EQ113}
  \end{align}
  for the stopping times $\tau^{k}$ in~\eqref{EQ105}, 
concluding the proof. 
\end{proof}

\begin{Remark}\label{R02}
Corresponding to Remark~\ref{R01}, when we relax the smallness assumption on $\epsilon_\sigma$ in~\eqref{EQ04}, 
we instead obtain $\PP(\tau<T)\lec_T p_0$ in Lemma~\ref{L10} since $C$ depends on~$T$. 
\end{Remark}

Next, we establish that $\tau$ is positive $\mathbb{P}$-almost surely, without requiring smallness of~$\epsilon_\sigma$. 

\cole
\begin{Lemma}
\label{L08}
Let $k\in\mathbb{N}_0$, and
assume that $\vk_0$ is an initial datum for~\eqref{EQ85} satisfying the assertions of Lemma~\ref{L05}.
With $\tau$ as defined in \eqref{EQ106}, we have
  \begin{align}
   \PP(\tau>0)
   = 1
   ,
   \label{EQ114}
  \end{align}
provided that $0<\bar\epsilon\ll 1$, $\epsilon_0$ is sufficiently small relative to $\bar{\epsilon}$, and $M_k$ is sufficiently large relative to $\MM_k$; see their definitions in~\eqref{EQ11}, \eqref{EQ30}, \eqref{EQ33}, and~\eqref{EQ34}.
\end{Lemma}
\colb

\begin{proof}[Proof of Lemma~\ref{L08}]
Note that the smallness assumption \eqref{EQ86} holds as a consequence of~\eqref{EQ31} and Lemma~\ref{L09}.
Hence, we proceed as in the proofs of Lemma~\ref{L07} and Lemma~\ref{LM04}, applying Lemma~\ref{L01} to the system~\eqref{EQ85} componentwise on $[0, t_0]$ to obtain
	\begin{align}
	\EE\biggl[
	\sup_{0\leq t\leq t_0}
	Q^2_{k,\alpha}(t)-C\Vert\vk_{0}\Vert_{H^{\frac12 +\alpha}}^2
	 \biggr]
	\leq C
	\epsilon_{\sigma}^2
	\EE\biggl[
	\int_0^{t_0}\psi_k^2 \phi_k^2 
	\Vert \vk\Vert_{H^{\frac12 + \alpha}}^2\,dt
	\biggr]
	\comma \alpha=0,\delta,
	\llabel{EQ115}
	\end{align}
	for some constant $C\geq1$ that is independent of both~$k$ and~$t_0$. We may assume $0<\epsilon_0\ll \bar\epsilon$ so that $\Vert v_{0}^{(k)}\Vert_{\LLLit}\leq \bar{\epsilon}/C4^{k+1}$. From this, Markov's inequality implies
	\begin{align}
	\begin{split}
	&
	\PP\biggl(
	\sup_{t\in [0, t_0]} Q_{k,0}(t)
	\geq  
	\frac{\overline\epsilon}{2^{k}}
	\biggr)
	\leq 
	 \PP\biggl(
	\sup_{t\in [0, t_0]} Q^2_{k,0}(t)-C\Vert\vk_0 \Vert_{H^\frac12}^2
	\geq \left(
	\frac{\overline\epsilon}{ 2^{k}}
	\right)^2
	-
	\left(
	\frac{\overline\epsilon}{4^{k+1}}
	\right)^2
	\biggr)
	\\&\indeq
	\lec 
	\left(
	\frac{\overline\epsilon}{2^{k}}
	\right)^{-2}
	\epsilon_{\sigma}^2
	\EE\biggl[
	\int_0^{t_0}\psi_k^2 \phi_k^2 
	\Vert \vk\Vert_{H^{\frac12 }}^2\,dt
	\biggr]
	\lec
	t_0 \left(
	\frac{\overline\epsilon}{2^{k}}
	\right)^{-2}
	\left(
	\frac{\overline\epsilon}{4^{k+1}}
	\right)^{2}
	\lec \frac{t_0}{2^{2k}}
	.
	\end{split}
	\label{EQ116}
	\end{align}
	Similarly, we may assume $ \MM_k\ll M_k$ so that
	\begin{align}
	\begin{split}
	&\PP\biggl(
	\sup_{t\in [0, \delta]} Q^2_{k,\delta}(t)
	\geq  M_{k}^2
	\biggr)
	\leq 
		\PP\biggl(
	\sup_{t\in [0, t_0]} Q^2_{k,\delta}(t)-C\Vert\vk_0 \Vert_{H^\frac12+\delta}^2
	\geq 
	\frac{M_k^2}{ 2}
	\biggr)
		\\&\indeq
	\lec 
	M_k^{-2}
	\epsilon_{\sigma}^2
	\EE\biggl[
	\int_0^{t_0}\psi_k^2 \phi_k^2 
	\Vert \vk\Vert_{H^{\frac12 +\delta}}^2\,dt
	\biggr]
	\lec
	t_0 
	\left(
	\frac{\MM_k}{M_k}
	\right)^{2}
	\lec
	\frac{ t_0}{2^{2k}}
	.
	\end{split}
	\label{EQ117}
	\end{align}
Therefore, for all $k\in\mathbb{N}_0$,
	\begin{align}
	\begin{split}
	\PP(\tau^k<t_0)
	&\leq
	\sum_{j=0}^{k}
	\mathbb{P}
	\Bigl(
	\sup_{0\leq t< t_0}Q_{j,\delta}(t)\geq M_j
	\Bigr)
	+
	\sum_{j=0}^{k}
	\mathbb{P}
	\Bigl(
	\sup_{0\leq t< t_0}Q_{j,0}\geq \frac{\bar{\epsilon}}{2^{j}}
	\Bigr)
	\lec
	t_0
	,
	\end{split}
	\llabel{EQ118}
	\end{align}
which leads to~\eqref{EQ114} and concludes the proof. 
\end{proof}

\startnewsection{Conclusion of the proofs of Theorems~\ref{T01} and~\ref{T02}}{sec.thm}
Suppose that $u_0$ satisfies~\eqref{EQ11}. Then, Lemma~\ref{L05} provides a sequence of functions $\{v_{0}^{(j)}\}_{k\in\NNz}$ that are divergence-free, have zero spatial-average, and satisfy~\eqref{EQ26}--\eqref{EQ28}, with their partial sums converging to~$u_0$ in $H^{\frac12}$ almost surely. 

We have proven that, for every $k\in\NNz$, the stochastic Navier-Stokes-like system~\eqref{EQ85} has a solution $v^{(k)}$ in $L^2(\Omega; C([0,T], H^{\frac{1}{2}+\delta}))\cap L^2(\Omega; L^2([0,T], H^{\frac{3}{2}+\delta}))$ for all $T>0$. Moreover, $v^{(k)}$ solves~\eqref{EQ32} up to a common stopping time $\tau$; see~\eqref{EQ106} for the definition. By Lemma~\ref{L08}, $\tau$ is positive $\PP$-almost surely if $0<\epsilon_0\ll \bar\epsilon\ll 1$ and $0<\MM_k\ll M_k$; see~\eqref{EQ11}, \eqref{EQ30}, \eqref{EQ33}, and~\eqref{EQ34} for their definitions. Additionally, if $\epsilon_{\sigma}$ is sufficiently small (recall the assumption~\eqref{EQ04}), then $\PP(\tau<\infty)\leq p_0$ for any $p_0 \in (0,1)$, with the relative smallness of $\epsilon_0$ to $\bar\epsilon$ and the relative largeness of $M_k$ to $\MM_k$
depending on~$p_0$. Without assuming that $\epsilon_{\sigma}$ is small, we then have $\PP(\tau<T)\leq p_0$, where the relative smallness of $\epsilon_0$ to $\bar\epsilon$ and the relative largeness of $M_k$ to $\MM_k$ depend on $p_0$ and~$T$. 

Now, we refer to~\eqref{EQ31} for the definition of $\uk$ and note that
  \begin{align}
  \begin{split}
   &
   \sum_{j=0}^{k}
   \Bigl(
    (
    \vj\cdot \nabla)\vj
    + (\ujm \cdot\nabla) \vj
    + (\vj \cdot\nabla) \ujm
   \Bigr)
   =
   (\uk \cdot\nabla) \uk\comma k\in\mathbb{N}_0.
  \end{split}
   \label{EQ119}
  \end{align}
Then, $(\uk,\tau)$ is a solution to
  \begin{align} 
  \begin{split}
    &\partial_t \uk - \Delta \uk  + \mathcal{P}((\uk\cdot\nabla) \uk)
    = \sigma(t,\uk) \dot{W}(t),
    \\
    &\nabla\cdot \uk = 0,
    \\
    & \uk(0) = \uk_0 = v_{0}^{(0)} + \cdots + v_{0}^{(k)}
    \comma k\in\mathbb{N}_0,
  \end{split}
  \label{EQ120}
  \end{align}
that is, for any $T>0$, $\phi\in C^{\infty}(\TT^3)$, and $t\in [0, T]$, 
  \begin{align}
   \begin{split}
     (\uk_j( t \wedge \tau),\phi)
     &=
     (u_{0,j}^{(k)},\phi)
   + \int_{0}^{t\wedge\tau}
         (\uk_j (r), \Delta\phi )
     \,dr
   \\&\indeq
   - \int_{0}^{t\wedge\tau}
      \bigl(\bigl(\mathcal{P}  (\uk_m(r) \uk(r))\bigr)_j ,\partial_{m} \phi\bigr)
     \,dr
   \\&\indeq
     +\int_0^{t\wedge\tau} \bigl(\sigma_j(r, \uk(r)),\phi\bigr) \,dW(r)\quad
     \PP\mbox{-a.s.}
    \comma j=1,2,3
    .
   \end{split}
   \label{EQ121}
  \end{align}
Recalling~\eqref{EQ102}, we obtain that $\uk$ has a limit $u\in L^\infty(\Omega;L^\infty([0,\infty), H^\frac12) \cap L^2([0,\infty), H^{\frac32}))$. This convergence and \eqref{EQ121} lead to
  \begin{align}
  \begin{split}
  (\uu_j( t \wedge \tau),\phi)
  &=
  (\uu_{j}(0),\phi)
  + \int_{0}^{t\wedge\tau}
  (\uu_j (r), \Delta\phi )
  \,dr
  \\&\indeq
  - \int_{0}^{t\wedge\tau}
  \bigl(\bigl(\mathcal{P}  (\uu_m(r) \uu(r))\bigr)_j ,\partial_{m} \phi\bigr)
  \,dr
  +\int_0^{t\wedge\tau} \bigl(\sigma_j(r, \uu(r)),\phi\bigr)\,dW(r)\quad
  \PP\mbox{-a.s.}
  \end{split}
  \llabel{EQ122}
  \end{align} 
for $ j=1,2,3$
and
  \begin{align}
		\int_{0}^{t\wedge\tau}
		(\uu (r) , \nabla \phi )\,dr=0\quad
		\PP\mbox{-a.s.}\comma t\in [0, T]
		,
	\llabel{EQ123}
\end{align} 
which gives that $(u,\tau)$ is a solution of~\eqref{EQ01}. 

Now we establish~\eqref{EQ12}. Applying Lemma~\ref{L01} to \eqref{EQ120} componentwise, as in the proof of Lemma~\ref{L07}, and using~\eqref{EQ102}, we obtain that
\begin{align}
  \begin{split}
  &\EE\biggl[\sup_{0\leq t\leq T\wedge\tau}\Vert\uk(t,\cdot)\Vert_{H^{\frac12}}^2
  +\int_0^{T\wedge\tau} \Vert\uk(t,\cdot)\Vert_{H^{\frac{3}{2}}}^2  \,dt\biggr]
  \\&\indeq
  \leq C
  \EE\biggl[
  \Vert\uk_{0}\Vert_{H^{\frac12}}^2
  +
  (\bar\epsilon+\epsilon_{\sigma}^2)\int_0^{T\wedge\tau} 
  \Vert \uk(t,\cdot)\Vert_{H^{\frac{3}{2}}}^2
  \,dt
  \biggr]\comma k \in \mathbb{N}_0,
  \end{split}
  \label{EQ124}
  \end{align}
for some constant $C$ that is independent of~$T$. Therefore, we have
\begin{align}
\begin{split}
\EE\biggl[
\sup_{0\leq t\leq \tau}\Vert\uk(t)\Vert_{H^{\frac12}}^2
+\int_0^{\tau} \Vert \uk\Vert_{H^{\frac{3}{2}}}^2 \,dt
\biggr]
\leq 
C\epsilon_0^{2}\comma k \in \mathbb{N}_0,
\end{split}
\llabel{EQ125}
\end{align}
upon choosing $\bar\epsilon$ and $\epsilon_{\sigma}$ sufficiently
small and letting $T\to\infty$. Then, utilizing the convergence of
$\uk$ and the lower-semicontinuity of the Hilbert norm under the weak
convergence, we may pass to the limit in the above inequality and conclude~\eqref{EQ12}. Note that, in this situation, $\PP(\tau<\infty)\leq p_0$ for a small~$p_0$.

The uniqueness of the solution follows by repeating the proof of Lemma~\ref{L07} for the difference between $u$ and the other solution.

Even without $\epsilon_\sigma$ being small, the limit function $u$ solves~\eqref{EQ01} up to $\tau$, and $\PP(\tau<T)\lec_T p_0$ for a small~$p_0$. Furthermore, we can write
 \begin{align}
\begin{split}
&\EE\biggl[\sup_{0\leq t\leq T\wedge\tau}\Vert\uk(t)\Vert_{H^{\frac{1}{2}}}^2
+\int_0^{\tau\wedge T}  \Vert\uk(t)\Vert_{H^{\frac{3}{2}}}^2 \,dt\biggr]
\\&\indeq
\leq C
\EE\biggl[
\Vert\uk_{0}\Vert_{H^{\frac{1}{2}}}^2
+
\bar\epsilon\int_0^{\tau\wedge T} 
\Vert \uk\Vert_{H^{\frac{3}{2}}}^2
\,dt
+\epsilon_{\sigma}^2
\int_0^{T}
\Vert \uk\Vert_{H^{\frac{1}{2}}}^2
\,dt
\biggr] \comma k \in \mathbb{N}_0,
\end{split}
   \llabel{EQ126}
\end{align}	
and then use Gr\"onwall's lemma to obtain
 \begin{align}
\begin{split}
\EE\biggl[
\sup_{0\leq t\leq \tau\wedge T}\Vert\uk(t)\Vert_{H^{\frac{1}{2}}}^2
+\int_0^{\tau\wedge T} 
  \Vert\uk(t)\Vert_{H^{\frac{3}{2}}}^2\,dt
\biggr]
\leq 
C\epsilon_0^{2}\comma k \in \mathbb{N}_0,
\end{split}
   \llabel{EQ127}
\end{align}
where $C>0$ depends on~$T$. Since $\uk$ converges in $L^\infty(\Omega;L^\infty([0,\infty), H^\frac12) \cap L^2([0,\infty), H^{\frac32}))$, 
upon passing to the limit, we achieve~\eqref{EQ16}. 

For this case, we first establish the pathwise uniqueness on a sufficiently small interval $[0, \tau\wedge t_0]$, from where we extend it to $[0, \tau\wedge T]$, for $T>0$, in finitely many steps.

\section*{Acknowledgments}
\rm
MSA and IK were supported in part by the NSF grant DMS-2205493.

\ifnum\sketches=1

\fi

\end{document}